\documentclass[11pt,leqno]{article}
\usepackage{amsthm,amsfonts,amssymb,amsmath,oldgerm}
\usepackage{epsfig}
\numberwithin{equation}{section}

\newcommand{\CalB}{\mathcal{B}}

\def\eps{\varepsilon }

\def\eps{\varepsilon}

\newcommand\br{\begin{remark}}
\newcommand\er{\end{remark}}
\newcommand\bp{\begin{pmatrix}}
\newcommand\ep{\end{pmatrix}}
\newcommand\be{\begin{equation}}
\newcommand\ee{\end{equation}}
\newcommand\ba{\begin{equation}\begin{aligned}}
\newcommand\ea{\end{aligned}\end{equation}}

\newcommand{\bap}{\begin{app}}
\newcommand{\eap}{\end{app}}
\newcommand{\begs}{\begin{exams}}
\newcommand{\eegs}{\end{exams}}
\newcommand{\beg}{\begin{example}}
\newcommand{\eeg}{\end{exaplem}}
\newcommand{\bpr}{\begin{proposition}}
\newcommand{\epr}{\end{proposition}}
\newcommand{\bt}{\begin{theorem}}
\newcommand{\et}{\end{theorem}}
\newcommand{\bc}{\begin{corollary}}
\newcommand{\ec}{\end{corollary}}
\newcommand{\bl}{\begin{lemma}}
\newcommand{\el}{\end{lemma}}
\newcommand{\bd}{\begin{definition}}
\newcommand{\ed}{\end{definition}}
\newcommand{\brs}{\begin{remarks}}
\newcommand{\ers}{\end{remarks}}

\newtheorem{theo}{Theorem}[section]

\newtheorem{exams}[theo]{Examples}

\numberwithin{equation}{section}

\newcommand{\CalL}{\mathcal{L}}

\newcommand{\Id}{{\rm Id }}

\newcommand{\sgn}{\text{\rm sgn}}
\newtheorem{theorem}{Theorem}[section]
\newtheorem{proposition}[theorem]{Proposition}
\newtheorem{corollary}[theorem]{Corollary}
\newtheorem{lemma}[theorem]{Lemma}
\newtheorem{definition}[theorem]{Definition}

\newtheorem{example}[theorem]{Example}
\newtheorem{remark}[theorem]{Remark}

\title{Existence and stability of viscoelastic shock profiles}

\author{\sc \small
Blake Barker
\thanks{Indiana University, Bloomington, IN 47405;
bhbarker@indiana.edu:
Research of B.B. was partially supported
under NSF grant no. DMS-0801745. }
Marta Lewicka
\thanks{Department of Mathematics, University of Minnesota
206 Church S.E., Minneapolis, MN 55455;
lewicka@math.umn.edu:
Research of M.L. was partially supported
under NSF grants no. DMS-0707275 and DMS-0846996. }
Kevin Zumbrun
\thanks{Indiana University, Bloomington, IN 47405;
kzumbrun@indiana.edu:
Research of K.Z. was partially supported
under NSF grants no. DMS-0300487 and DMS-0801745. }}

\begin{document}

\maketitle


\begin{abstract}
We investigate existence and stability of
viscoelastic shock profiles for a class of planar models
including the incompressible shear case
studied by Antman and Malek-Madani.
We establish that the resulting equations fall
into the class of symmetrizable hyperbolic--parabolic systems, 
hence spectral stability implies linearized
and nonlinear stability with sharp rates of decay.
The new contributions are treatment
of the compressible case,
formulation of a rigorous nonlinear stability theory,
including verification of stability of small-amplitude Lax shocks,
and the systematic incorporation in our investigations of
numerical Evans function computations determining stability of
large-amplitude and or nonclassical type shock profiles.
\end{abstract}



\section{Introduction}\label{s:intro}

In this paper, generalizing work of Antman and Malek--Madani \cite{AM}
in the incompressible shear flow case, we carry out the numerical and
analytical study of the existence and stability of
planar viscoelastic traveling waves in a $3$d solid,
for a simple prototypical elastic energy density,
both for the general compressible and the incompressible shear flow
case. We establish that the resulting equations fall
into the class of symmetrizable hyperbolic--parabolic systems studied
in \cite{MaZ2,MaZ3,MaZ4,RZ,Z4}, hence spectral stability implies linearized
and nonlinear stability with sharp rates of decay.
This important point was previously left undecided, due to a lack
of the necessary abstract stability framework.

The new contributions beyond what was done in \cite{AM} are: treatment
of the compressible case, consideration of large-amplitude waves
(somewhat artificial given our simple choice of
energy density; however, the methods used clearly generalize to more physically
correct models), 
formulation of a rigorous nonlinear stability theory including verification
of stability of small-amplitude Lax waves,
and the systematic incorporation in our investigations of
numerical Evans function computations determining stability of
large-amplitude and or nonclassical type shock profiles.
For related analysis in various different settings,
see \cite{BHRZ,HLZ,HLyZ,CHNZ,BHZ,BLZ}.

In preparation for future generalizations, we also
discuss the case of phase-transitional viscoelasticity.
It would be interesting to carry out similar analysis for
more general classes of elastic energy density, as well as for the 
phase-transitional case which involves, at the
technical level higher order dispersive terms relating to surface
energy, and at the physical
level, presumably, interesting new behaviors.


\bigskip
{\bf Acknowledgment.} Thanks to Stuart Antman and Constantine
Dafermos for several helpful conversations, and to Stuart
Antman for making available the working notes \cite{A}.

\section{The equations of viscoelasticity}\label{s:theory}

The equations of isothermal viscoelasticity are given through the following
balance of linear momentum:
\begin{equation}\label{visco_eq}
\xi_{tt} - \nabla_X\cdot \Big(DW(\nabla\xi) + \mathcal{Z}(\nabla\xi,
\nabla\xi_t)\Big)=0.
\end{equation}
Here, $\xi:\Omega\times\mathbb{R_+}\longrightarrow \mathbb{R}^3$
denotes the deformation of a reference configuration
$\Omega\subset\mathbb{R}^3$ which models
a viscoelastic body with constant temperature and density.
A typical point in $\Omega$ is denoted by $X$, so that
the deformation gradient is given as:
$$F = \nabla\xi\in\mathbb{R}^{3\times 3},$$
with the key constraint of:
$$\det F>0.$$
In (\ref{visco_eq}) the operator
$\nabla_X\cdot$ stands for the divergence of an appropriate field.
We use the convention that the divergence of a matrix field is taken
row-wise. In what follows, we shall also use the matrix norm
$|F|=(\mbox{tr}(F^TF))^{1/2}$, which
is induced by the inner product: $F_1:F_2=\mbox{tr}(F_1^TF_2)$.

The mapping $DW:\mathbb{R}^{3\times 3}\longrightarrow \mathbb{R}^{3\times 3}$
is the Piola-Kirchhoff stress tensor which, in agreement with
the second law of thermodynamics,
is expressed as the derivative of an elastic energy density
$W:\mathbb{R}^{3\times 3}\longrightarrow \overline{\mathbb{R}}_+$.
The viscous stress tensor is given by the mapping
$\mathcal{Z}:\mathbb{R}^{3\times 3}\times \mathbb{R}^{3\times 3}
\longrightarrow \mathbb{R}^{3\times 3}$, depending on the deformation gradient $F$
and the velocity gradient $Q=F_t = \nabla\xi_t=\nabla v$, where $v=\xi_t$.

The first order version of the inviscid part of  (\ref{visco_eq}):
\be\label{fo}
 \xi_{tt} - \nabla_X\cdot \Big(DW(\nabla\xi) \Big)=0
\ee
is:
\begin{equation}\label{firstorder}
(F,v)_t + \sum_{i=1}^3 \partial_{X_i}\big(\tilde G_i(F,v)\big) = 0.
\end{equation}
Above, $(F,v) :\Omega\longrightarrow \mathbb{R}^{12}$ represents
 conserved quantities, while
$\tilde G_i:\mathbb{R}^{12}\longrightarrow \mathbb{R}^{12}$ given by:
$$-\tilde G_i(F,v) = v^1e_i \oplus   v^2e_i \oplus  v^3e_i \oplus
\left[\frac{\partial}{\partial F_{ki}}W(F)\right]_{k=1}^3, \quad i=1..3$$
are the fluxes, and
$e_i$ denotes the $i$-th coordinate vector of $\mathbb{R}^3$.

\subsection{The elastic energy density $W$}

The principle of material frame invariance imposes the following condition on
$W$, with respect to the group $SO(3)$ of proper rotations in $\mathbb{R}^3$:
\begin{equation}\label{frame_inv}
W(RF) = W(F) \qquad \forall F\in\mathbb{R}^{3\times 3} \quad\forall R\in SO(3).
\end{equation}
Also, the material consistency requires that:
\begin{equation}\label{mat_inc}
W(F)\to +\infty \quad \mbox{ as } \det F\to 0.
\end{equation}
In what follows, we shall restrict our attention to the class of isotropic
materials,
for whom the energy $W$ satisfies additionally:
\begin{equation}\label{isotropic}
W(FR) = W(F) \qquad \forall F\in\mathbb{R}^{3\times 3} \quad\forall R\in SO(3).
\end{equation}
Recall \cite{Ball} that hyperbolicity of (\ref{fo})
is equivalent to rank-one convexity of $W$.

\medskip

A particular example of $W$ satisfying (\ref{frame_inv}) and (\ref{isotropic})
is:
\begin{equation}\label{part}
W_0(F)=\frac{1}{4}|F^TF-\mbox{Id}|^2 = \frac{1}{4}(|F^T F|^2 - 2|F|^2+3)
\end{equation}
and by a direct calculation, we obtain:
$$DW_0(F) = F(F^TF-\mathrm{Id}).$$
Note that $W_0$ in (\ref{part}) is not quasiconvex (or polyconvex) as
it is not globally rank-one convex. This follows by checking the
Legendre-Hadamard condition. Indeed, for any $A\in\mathbb{R}^{3\times 3}$
one has: $\partial_{AA}^2 W_0(F) = (FA^TF + FF^TA + AF^TF - A):A$.
Taking $A=\mbox{Id}$ and $F\in\mbox{skew}$ we obtain
$\partial_{AA}^2W_0(F) = |F|^2-3$ which is negative for $|F|<\sqrt{3}$.

On the other hand we see that $\partial_{AA}^2 W_0(R)
= 2|\mbox{sym}(AR^T)|^2$ for $R\in SO(3)$.  If $\mbox{rank } A=1$
then $\mbox{rank} (AR^T)=1$ so $\mbox{sym}(AR^T)\neq 0$.
Therefore $\partial_{AA}^2 W_0(R)\geq c|A|^2$ for every $R\in SO(3)$
and every rank-one matrix $A$, with a uniform $c>0$.
This implies that $W_0$ is rank-one convex in a
neighborhood of $SO(3)$.

Notice also that $W_0$ has quadratic growth close to $SO(3)$.
Indeed, write $F=R+ E$, where for $F$ close to $SO(3)$ we have:
$R=\mathbb{P}_{SO(3)}F$ and $|E| = \mbox{dist}(F, SO(3))$.
Since $E$ is orthogonal to  the tangent space to $SO(3)$ at $R$, we
see that $R^TE$ must be symmetric.
Therefore: $ W_0(F) =\frac{1}{4}|R^TE + E^TR + E^TE|^2 = \frac{1}{4}|2R^TE + E^TE|^2
= |E|^2 +\mathcal{O}(|E|^3).$

For other examples of $W$ satisfying (\ref{frame_inv}) and (\ref{isotropic}), see
\eqref{genform} and Appendix \ref{s:general}.

\subsection{The viscous stress tensor  $\mathcal{Z}$}

The viscous stress tensor $\mathcal{Z}:\mathbb{R}^{3\times 3}\times
\mathbb{R}^{3\times 3}
\longrightarrow \mathbb{R}^{3\times 3}$ should be compatible with the
following principles of continuum mechanics:
balance of angular momentum, frame invariance,
and the Claussius-Duhem inequality. That is, for every $F,Q\in
\mathbb{R}^{3\times 3}$
with $\det F\neq 0$, we require that:
\begin{equation}\label{visc_tensor}
\begin{minipage}{14cm}
\begin{itemize}
\item[(i)] $\mbox{skew}\left(F^{-1}\mathcal{Z}(F,Q)\right) = 0$,
i.e. $\mathcal{Z}=FS$ with $S$ symmetric.
\item[(ii)] $\mathcal{Z}(RF,R_tF + RQ) = R\mathcal{Z}(F,Q)$
for every path of rotations $R:\mathbb{R}_+\longrightarrow SO(3)$,
i.e. in view of (i): $S(RF,RKF+RQ)=S(R,Q)$ $\forall R\in SO(3)$
$\forall K\in \mbox{skew}$.
\item[(iii)] $\mathcal{Z}(F,Q):Q\geq 0$,
i.e. in view of (i): $S: \mbox{sym}(F^T Q) \geq 0$.
\end{itemize}
\end{minipage}
\end{equation}
Examples of $\mathcal{Z}$ satisfying the above are:
\begin{equation}\label{Zs}
\begin{split}
\mathcal{Z}_1(F,Q) &= 2F\mbox{sym}(F^TQ),\\
\mathcal{Z}_2(F,Q) &= 2(\mbox{det} F) \mbox{sym}(QF^{-1}) F^{-1, T}.\\
\end{split}
\end{equation}
We note that in the case of $\mathcal{Z}_2$, the related Cauchy stress tensor
$T_2 = 2(\mbox{det} F)^{-1} \mathcal{Z}_2 F^T = 2\mbox{sym}(QF^{-1})$
is the Lagrangian version of the stress tensor  $2\mbox{sym}\nabla v$ written in
Eulerian coordinates. 
For incompressible fluids $2\mbox{div}(\mbox{sym}\nabla v) = \Delta v$,
giving the usual parabolic viscous regularization of the fluid
dynamics evolutionary system.
For more general viscous stress tensors, see
Appendix \ref{s:genvisc}.

\subsection{An extension: the surface energy}\label{s:surface}
A phenomenological modification that is sometimes used is to
replace \eqref{visco_eq} with
\begin{equation}\label{surface_eq}
\xi_{tt} - \nabla_X\cdot \Big(DW(\nabla\xi) + \mathcal{Z}(\nabla\xi,
\nabla \xi_t) - \mathcal{E}(\nabla^2 \xi)\Big)=0,
\end{equation}
where the surface energy $\mathcal{E}$ is given by:
$$\mathcal{E}(\nabla^2\xi) = \nabla_X\cdot D\Psi(\nabla^2\xi) 
= \left[\sum_{i=1}^3 \frac{\partial}{\partial X_i}
\left(\frac{\partial}{\partial (\partial_{ij}\zeta^k)}\Psi(\nabla^2\xi)
\right) \right]_{j,k:1\ldots 3}$$
for some convex density 
$\Psi:\mathbb{R}^{3\times 3\times 3}\longrightarrow \mathbb{R}$, 
compatible with frame indifference (and isotropy).
A typical example is $\Psi_0(G)=\frac{1}{2}|G|^2$, so that:
\begin{equation}\label{exentropy}
 \mathcal{E}_0(\nabla^2\xi)=\nabla_X \cdot \nabla^2\xi= \Delta_X F
\end{equation}
which is an extension of the 1d case of \cite{Sl}.

Writing the variation of the energy $\int \Psi(\nabla^2\xi)$
in the direction of a test function
$\phi\in\mathcal{C}_c^\infty(\Omega, \mathbb{R}^3)$ we obtain:
\begin{equation}\label{help}
\int_\Omega D\Psi(\nabla^2\xi):\nabla^2\phi 
= \int_\Omega \big(\nabla_X\cdot \mathcal{E}(\nabla^2\xi)\big)\cdot \phi,
\end{equation}
which justifies the last divergence term in (\ref{surface_eq}).

The addition of surface energy is motivated by the
van der Waals/Cahn--Hilliard approach to the stationary equilibrium theory
\cite{Sl,Z8,CGS,SZ}.
This would be an interesting direction for further investigation.

\subsection{Entropy}
After integrating (\ref{visco_eq}) against $\xi_t$ on $\Omega$, and
then integrating by parts, we obtain:
$$\frac{1}{2}\int\partial_t|\xi_t|^2 + \int \Big(DW(\nabla\xi) +
\mathcal{Z}(\nabla\xi, \nabla\xi_t)\Big):\nabla \xi_t = 0,$$
where we used that:
$$\xi_t^T \big(DW(\nabla\xi) + \mathcal{Z}(\nabla\xi,
\nabla\xi_t)\big)\vec n = 0 \quad \mbox{ on } \partial\Omega,$$
a natural assumption following from either (Dirichlet)
clamped boundary conditions $\xi_{|\partial \Omega}= const$
or else free (Neumann) conditions 
$\big(DW(\nabla\xi) + \mathcal{Z}(\nabla\xi, \nabla\xi_t)\big)_{|\partial
\Omega}= 0$ corresponding to the absence of stress on the boundary.

Consequently:
$$\partial_t\int\left(\frac{1}{2}\partial_t|\xi_t|^2 + DW(\nabla\xi)\right) 
= - \int \mathcal{Z}(\nabla\xi, \nabla\xi_t):\nabla\xi_t\leq 0$$
by the Clausius--Duhem inequality, and we see that the integral
$\int \eta$ of the quantity:
\be\label{entr}
\eta(F, v)= \frac{1}{2}|v|^2 + W(F)
\ee
along $F=\nabla\xi$ and $v=\xi_t$, is nonincreasing in time.
In case of (\ref{surface_eq}), using (\ref{help}) with $\phi=\xi_t$ we
obtain that $\int\tilde \eta$ is nonincreasing, for:
$$\tilde \eta= \frac{1}{2}|\xi_t|^2 + W(\nabla \xi) + D\Psi(\nabla^2 \xi).$$

Further, notice that $\eta:\mathbb{R}^{12}\longrightarrow \mathbb{R}$
defined in (\ref{entr}) is an entropy \cite{dafermos_book} 
associated to (\ref{firstorder}).  Indeed, the scalar fields:
$$q_i(F,v) = - v\cdot \left[\frac{\partial}{\partial
    F_{ik}}W(F)\right]_{k=1}^3 \qquad i=1..3$$
define the respective entropy fluxes, in the sense that:
$$\nabla q_i(F,v) = \nabla\eta(F,v) D\tilde G_i(F,v).$$

\section{The planar case}
We now restrict our attention 
to the interesting subclass of planar solutions,
which are solutions in full 3d space that depend only on a
single coordinate direction.
Namely, we  assume that the deformation $\xi$ has the form:
$$\xi(X) = X+ U(z), \qquad X=(x,y,z), \quad U=(u,v,w)\in\mathbb{R}^3,$$
which yields the following structure of the deformation gradient:
\begin{equation}\label{Fplanar}
F = \left[\begin{array}{ccc}1 &0&u_z\\
0&1&v_z\\
0&0&1 + w_z\end{array}\right]
= \left[\begin{array}{ccc}1 &0 & a_1\\
0&1&a_2\\
0&0&a_3\end{array}\right].
\end{equation}
We shall denote $V= (a,b) = (a_1, a_2, a_3, b_1, b_2, b_3)$, where
$a_1 = u_z, a_2 = v_z, a_3 = 1+ w_z$ and $b_1 = u_t,
b_2 = v_t, b_3 = w_t$, with the constraint: 
\be\label{feas}
a_3>0,
\ee
corresponding to $\det F>0$ in the region of physical feasibility of $V$.

Writing $W(a) = W(\left[\begin{array}{ccc}1 &0&a_1\\
0&1&a_2\\ 0&0&a_3\end{array}\right])$, we see that for all $F$ as in
(\ref{Fplanar}) there holds:
$$\nabla_X\cdot (DW(F)) = (D_a W(a))_z.$$
That is, the planar equations inherit a vector-valued 
variational structure echoing the matrix-valued variational
structure, and thus (\ref{fo}) has the
following form:
$$V_t + G(V)_z = 0$$
\begin{equation}\label{dg}
G(V) = (-b, -D_aW(a))^T, \qquad 
DG(V) = \left[\begin{array}{cc}0&-\mbox{Id}_3\\
-M&0\end{array}\right], \qquad M = D^2_aW(a).
\end{equation}
It follows that strict
hyperbolicity of (\ref{dg}) is equivalent to strict convexity of $W$
with $M$ having 3 distinct (positive) eigenvalues.
Also, $\eta(V)=\frac{1}{2}|b|^2+W(a)$ is then a convex entropy:
\begin{equation}\label{en_flu}
\nabla\eta (V) DG(V)=\nabla q(V), \qquad q(V)=-b\cdot D_aW(a).
\end{equation}

\subsection{Energy density $W_0$ and viscosities $\mathcal{Z}_i$}
By a straightforward calculation, we have:
\begin{equation}\label{w0}
W_0(a) = \frac{1}{4}(|a|^2-1)^2 + \frac{1}{2}(a_1^2 + a_2^2) 
\end{equation}
and:
\begin{equation}\label{wcalc}
\begin{split}
\mbox{div}(DW_0(F)) &=
 \Big(u_{zz} + (u_z|U_z|^2 + 2u_zw_z)_z,
~ v_{zz} + (v_z|U_z|^2 + 2v_zw_z)_z,\\
&\qquad\quad
~ 2w_{zz} + (|U_z|^2 + w_z|U_z|^2 + 2w_z^2)_z\Big)^T\\
&=\Big( (|a|^2a_1)_z , ~(|a|^2a_2)_z, ~((|a|^2-1)a_3)_z\Big)^T,
\end{split}
\end{equation}

More generally, one may consider densities of the form:
\be\label{genform}
 W(F)=\frac{1}{4} |F^TF - \mathrm{Id}|^2 + c_2(|F|^2-3)^2 +c_3(|\mathrm{det} F|- 1)^2.
\ee
The $c_2$ term contributes to $DW(F)$ as:
$4c_2(|F|^2-3)F$ in the planar case; this is
$4c_2(2w_z+|U_z|^2)F$, with divergence:
$$ 
\begin{aligned}
&4c_2\Big((u_z|U_z|^2 + 2u_zw_z)_z, (v_z|U_z|^2 + 2v_zw_z)_z,
(|U_z|^2+2w_z + w_z|U_z|^2 + 2w_z^2)_z\Big)^T  \\
&\qquad \qquad
\qquad \qquad
= 4c_2\Big(  ((|a|^2-1)a_1)_z, ((|a|^2-1)a_2)_z, ((|a|^2-1)a_3)_z \Big)^T. 
\end{aligned}
$$
The $c_3$ term contributes to $DW(F)$ the term:
$2(\mathrm{det} F-1)\mathrm{cof} F$, for $F$ with $\mathrm{det} F>0$.
In the planar case, divergence of this term reads 
$2c_3(0,0,w_{zz})^T=
2c_3(0,0,(a_3)_{z})^T .$

Combining, we obtain the general form:
\be\label{gendiv}
\mbox{div}(DW(F)) =
 \Big(  ((\mu_1 |a|^2 +\mu_2)a_1)_z, 
((\mu_1 |a|^2 +\mu_2)a_2)_z, 
((\mu_1 |a|^2 +\mu_2)a_3 + (\mu_3-1) a_3)_z \Big)^T, 
\ee
with $\mu_1=1+4c_2$, $\mu_2=-4c_2$, $\mu_3=2c_3$,
corresponding to elastic potential
\begin{equation*}
W(a)=  \frac{1}{4} \mu_1 |a|^4 + \frac{1}{2} \mu_2 |a|^2 + \frac{1}{2}
\mu_3 (a_3-1)^2 +C,
\end{equation*}
where $C$ is a constant.
The above potential is strictly convex at the identity ($a=(0,0,1)$)
whenever $\mu_1+\mu_2>0$ and $3\mu_1+\mu_2 + \mu_3 > 0$.
It is a simple case of the general form $W(a)=\tilde \sigma(|a|^2,a_3)$
described in Appendix \ref{s:general}.
Restricted to the incompressible planar case of Section \ref{incomp},
\eqref{gendiv} recovers the class of equations studied in \cite{AM}.

\medskip

Regarding the viscous tensors, we obtain:
\begin{equation}\label{z1calc}
\begin{split}
&\mbox{div}\big(\mathcal{Z}_1(F, \dot F)\big)
=  \Big(u_{zzt} + (u_z(2w_{zt}+(|U_z|^2)_t))_z,
~v_{zzt}+ (v_z(2w_{zt}+ (|U_z|^2)_t))_z,\\
&\qquad\qquad\qquad\qquad\quad
~2w_{zzt}+ (|U_z|^2)_{zt} + (w_z(2w_{zt}+(|U_z|^2)_t))_z\Big)^T\\
&= \left(\Big(b_{1,z}+2a_1 a\cdot b_z\Big)_z, ~\Big(b_{2,z}+ 2a_2
a\cdot b_z\Big)_z, ~\Big( 2a_3 a\cdot b_z\Big)_z\right)^T,
\end{split}
\end{equation}
\begin{equation}\label{z2calc}
\begin{split}
\mbox{div}\big(\mathcal{Z}_2(F, \dot F)\big)&
=  \left(\Big(\frac{u_{zt}}{1+w_z}\Big)_z,
~\Big(\frac{v_{zt}}{1+w_z}\Big)_z,
~\Big(\frac{2w_{zt}}{1+w_z}\Big)_z\right)^T\\
&= \left(\Big(\frac{b_{1,z}}{a_3}\Big)_z,
~\Big(\frac{b_{2,z}}{a_3}\Big)_z, ~2\Big(\frac{b_{3,z}}{a_3}\Big)_z\right)^T,
\end{split}
\end{equation}
where $a\cdot b_z = a_1 b_{1,z} + a_2 b_{2,z}+ a_3 b_{3,z}$.

\medskip

Hence, (\ref{visco_eq}) with $W$ and $\mathcal{Z}$ as in 
(\ref{part}) and (\ref{Zs}), has the hyperbolic-parabolic
form:
\begin{equation}\label{hyppar}
V_t + G(V)_z = (B(V)V_z)_z
\end{equation}
with $G(V)$ as in (\ref{dg}) and:
$$M=D_a^2W_0 = \mbox{diag}\Big(|a|^2, |a|^2, |a|^2-1\Big) + 2a\otimes a$$
in view of (\ref{w0}). Further:
\begin{equation}\label{blah}
B=\left[\begin{array}{cc} 0&0\\
0& B_{0,i}
\end{array}\right],\quad
B_{0,1} = \mbox{diag}\Big(1, 1, 0\Big) + 2a\otimes a
~~\mbox{ or } ~~ B_{0,2}= \frac{1}{a_3}\mbox{diag}\Big(1, 1, 2\Big)
\end{equation}
in case of $\mathcal{Z}_1$ and $\mathcal{Z}_2$, respectively. 
Both tensors $B_{0,i}$ are  symmetric and positive definite on
the entire physical region $a_3>0$.

\subsection{The full and the restricted systems in hyperbolic--parabolic form}\label{s:hp}

\subsubsection{Compressible viscoelasticity}
For the viscous stress tensor
$\mathcal{Z}_1$, system \eqref{hyppar} reads:
\begin{equation}\label{bfirstorder}
\begin{array}{l} a_{1,t} - b_{1,z}=0,\\
a_{2,t} - b_{2,z}=0,\\
a_{3,t} - b_{3,z}=0, \end{array}\qquad\quad
\begin{array}{l} \displaystyle{ b_{1,t} -  (|a|^2 a_1)_z =
\left(b_{1,z} + 2a_1 a\cdot b_z\right)_z}\\
 \displaystyle{b_{2,t} -  (|a|^2 a_2)_z =
\left(b_{2,z} + 2a_2 a\cdot b_z\right)_z}\\
 \displaystyle{b_{3,t} -  ((|a|^2-1)a_3)_z=
\left(2a_3 a\cdot b_z\right)_z.}\end{array}
\end{equation}
For the viscous tensor $\mathcal{Z}_2$ we have:
\begin{equation}\label{bfirstorder2}
\begin{array}{l} a_{1,t} - b_{1,z}=0,\\
a_{2,t} - b_{2,z}=0,\\
a_{3,t} - b_{3,z}=0, \end{array}\qquad\quad 
\begin{array}{l} \displaystyle{b_{1,t} -  (|a|^2 a_1)_z =\left(\frac{b_{1,z}}{a_3}\right)_z,}\\
b_{2,t} -  (|a|^2 a_2)_z =  \displaystyle{\left(\frac{b_{2,z}}{a_3}\right)_z,}\\
b_{3,t} -  ((|a|^2-1)a_3)_z =  \displaystyle{2\left(\frac{b_{3,z}}{a_3}\right)_z.}\end{array}
\end{equation}

\subsubsection{The 2D incompressible shear case}\label{incomp}
For an incompressible medium and a shear deformation where
$w =  0$, the system (\ref{bfirstorder2}) reduces to the following one
(naturally, we now denote $a=(a_1, a_2) $ and $|a|^2 = a_1^2 + a_2^2$):
\begin{equation}\label{shearfirstorder}
\begin{array}{l} a_{1,t} - b_{1,z}=0,\\
a_{2,t} - b_{2,z}=0,\end{array}\qquad\quad
\begin{array}{l} \displaystyle{b_{1,t} -  ((|a|^2+1)a_1)_z = b_{1,zz},}\\
 \displaystyle{b_{2,t} -  ((|a|^2+1)a_2)_z = b_{2,zz},}\end{array}
\end{equation}
with an associated pressure of $p=|a|^2$ whose gradient
$(0,0, (|a|^2)_z)^T$
cancels the term $-(|a|^2)_z$ in the $b_3$ equation
of \eqref{bfirstorder2}.
Note that the viscous stress tensor in this case reduces to the Laplacian.
Equations \eqref{shearfirstorder} are a special case of the equations
studied in \cite{AM}; they may be also recognized as the model
for an elastic string.

For the choice $\mathcal{Z}_1$, we obtain:
\begin{equation}\label{shearfirstorderZ2}
\begin{array}{l} a_{1,t} - b_{1,z}=0,\\
a_{2,t} - b_{2,z}=0,\end{array}\qquad\quad
\begin{array}{l} \displaystyle{b_{1,t} -  ((|a|^2+1)a_1)_z = (b_{1,z}+2a_1a\cdot b_z)_z,}\\
 \displaystyle{b_{2,t} -  ((|a|^2+1)a_2)_z = (b_{2,z}+2a_2a\cdot b_z)_z.}
\end{array}
\end{equation}
The incompressible model may be viewed as the formal limit
as $\mu\to +\infty$ of a system with potential $W_0(a) +\mu_3 (a_3-1)^2$, 
penalizing variations in density $\det F=a_3$.
Operationally, this amounts to fixing $a_3= 1$ in a given (compressible)
elastic potential and dropping the equation for $a_3$, to obtain
a reduced shear potential $\check W(a_1,a_2)=W(a_1,a_2,1)$ and equations
whose first-order part have the same variational structure \eqref{dg}
as the full 3d system.

\subsubsection{The 2D compressible case}\label{2D}
Another reduced version of \eqref{bfirstorder2}, 
restricted to the $v-w$ plane is obtained by setting $u=0$.
This is an equally simple system as (\ref{shearfirstorder}),
but with essentially different structure
(we now write $|a|^2 = a_2^2 + a_3^2$):
\begin{equation}\label{bfirstorder2D}
\begin{array}{l}a_{2,t} - b_{2,z}=0,\\
a_{3,t} - b_{3,z}=0,\end{array}\qquad\quad
\begin{array}{l} \displaystyle{b_{2,t} -  (|a|^2a_2)_z =\left(\frac{b_{2,z}}{a_3}\right)_z,}\\
 \displaystyle{b_{3,t} -  ((|a|^2-1)a_3)_z = 2\left(\frac{b_{3,z}}{a_3}\right)_z,}
\end{array}
\end{equation}
while for $\mathcal{Z}_1$, writing $a\cdot b_z = a_2b_{2,z} + a_3
b_{3,z}$, we have:
\begin{equation}\label{bfirstorder2DZ2}
\begin{array}{l} a_{2,t} - b_{2,z}=0,\\
a_{3,t} - b_{3,z}=0,\end{array}\qquad\quad
\begin{array}{l} \displaystyle{b_{2,t} -  (|a|^2a_2)_z =\left(b_{2,z} + 2a_2 a\cdot b_z\right)_z,}\\
 \displaystyle{b_{3,t} -  ((|a|^2-1)a_3)_z = \left(2a_3 a\cdot b_z\right)_z.}\end{array}
\end{equation}

\subsubsection{The 1D cases}
Taking $v=w= 0$, (\ref{bfirstorder2})  further reduces to
a model of the transverse unidirectional perturbations in a beam or
string:
\begin{equation}\label{shearfirstorder1D}
a_{1,t} - b_{1,z}=0,\qquad\quad
b_{1,t} -  (a_1^3+a_1)_z = b_{1,zz}.\\
\end{equation}

Setting $u=v=0$, (\ref{bfirstorder2}) yields the 1D compressible
model for longitudinal perturbations in a viscoelastic rod:
\begin{equation}\label{simplebfirstorder1D}
a_{3,t} - b_{3,z}=0,\qquad\quad
b_{3,t} -  (a_3^3-a_3)_z = 2\left(\frac{b_{3,z}}{a_3}\right)_z.
\end{equation}

\subsubsection{Extension: surface energy and higher-order dispersion}
We mention briefly the effects of modifying by the addition
of surface energy term.
In the planar, incompressible shear case, (\ref{surface_eq}) with
(\ref{exentropy}) becomes:
\begin{equation}\label{shearfirstordermod}
\begin{array}{l} a_{1,t} - b_{1,z}=0,\\
a_{2,t} - b_{2,z}=0,\end{array}\qquad\quad
\begin{array}{l} 
\displaystyle{b_{1,t} -  (a_1 +( a_1^2+a_2^2) a_1 )_z = b_{1,zz}-a_{1,zzz},}\\
\displaystyle{b_{2,t} -  (a_2 +( a_1^2+a_2^2) a_2 )_z = b_{2,zz}-a_{2,zzz}.}
\end{array}
\end{equation}
See \cite{Sl} for a corresponding treatment of the one-dimensional case.

\subsection{Hyperbolic characteristics}

\subsubsection{Compressible case}
Consider the inviscid version of (\ref{hyppar}): $V_t+DG(V)V_z = 0$.
Using the block structure of $DG$ in (\ref{dg}) 
we obtain that its eigenvalues are
$\{\pm\sqrt{m_j}\}_{j=1}^3$ with corresponding eigenvectors
$(\{r_j, ~\mp\sqrt{m_j}r_j)\}_{j=1}^3$, where $m_j$ (and $r_j$) are the
eigenvalues (and corresponding eigenvectors) of the symmetric
matrix $M$. Also, since $m_j$ are
independent of $b$, the linear degeneracy or genuine nonlinearity of
the $\pm\sqrt{m_j}$ characteristic fields of $DG$ is equivalent to the
same properties of the $m_j$ characteristic fields of $M$.

Using now the following formula, valid for $3\times 3$ matrices: 
$\mbox{det}(A+B) = \mbox{det}A + (\mbox{cof}A):B
+ (\mbox{cof}B):A + \mbox{det}B $, we obtain:
\begin{equation}
\begin{split}
m_1 &= |a|^2,\\
m_2 &= \frac{1}{2}\left(4|a|^2 - 1 - \sqrt{(2|a|^2-1)^2 + 8(a_1^2+a_2^2)}\right),\\
m_3 &= \frac{1}{2}\left(4|a|^2 - 1 + \sqrt{(2|a|^2-1)^2 + 8(a_1^2+a_2^2)}\right).
\end{split}
\end{equation}
Note that at $a=(0,0,1)$ we have $m_1=m_2=1$ and $m_3=2$; hence 
$DG$ is (nonstrictly) hyperbolic at $V_0=(0,0,0,b_1, b_2, b_3)$.
Further calculations show that, whenever defined:
\begin{itemize}
\item[(i)] $r_1 = (a_2, ~-a_1,~0)^T$ and the eigenvalues $\pm \sqrt{m_1}$
  correspond to two linearly degenerate fields of $DG$.
\item[(ii)] $r_2 = (-2a_1a_3,~-2a_2a_3, ~3|a|^2-2a_3^2-m_2)^T$ 
\item[(iii)] $r_3 =  (-2a_1a_3,~-2a_2a_3, ~3|a|^2-2a_3^2-m_3)^T$ and
  in the vicinity of $V_0$ the eigenvalues $\pm \sqrt{m_3}$
  correspond to two genuinely nonlinear fields of $DG$.
\end{itemize}

\subsubsection{The 2D incompressible shear case}

The system (\ref{shearfirstorder}) can be written as:
\begin{equation*}
a_t - b_z = 0, \qquad\quad
b_t - (D_aW(a))_z = b_{zz}.
\end{equation*}
Its flux matrix depends on $a=(a_1, a_2)$ and has the form:
\be\label{DG12}
DG_{1-2}= \bp 0 & -\mbox{Id}_2\\
-M_{1-2}& 0 \ep,
\qquad
M_{1-2}= D^2_aW_0(a) = (|a|^2+1) \mbox{Id}_2 + 2a\otimes a,
\ee
where $W_0(a) = \frac{1}{4} |a|^4 + \frac{1}{2}|a|^2$. We see that
strict hyperbolicity, convexity of $W_0$ and existence of strictly
convex entropy are equivalently satisfied here.

Calculating as before,
$DG_{1-2}$ has two genuinely nonlinear characteristic fields,
with eigenvalues $\pm\sqrt{1+3|a^2|}$ and corresponding
eigenvectors 
$$
(\pm a_1, \pm a_2, a_1\sqrt{1+3|a|^2}, a_2\sqrt{1+3|a|^2})^T
$$
in fast modes,
and two linearly degenerate fields with eigenvalues $\pm\sqrt{1+|a^2|}$ 
and eigenvectors $(\pm a_2, \mp a_1, a_2\sqrt{1+|a|^2}, -a_1\sqrt{1+|a|^2})^T$ 
in slow modes. The linear degeneracy reflects 
the rotational degeneracy of the underlying system \cite{F} .

\subsubsection{The 2D compressible case}\label{333}
The flux matrix in (\ref{bfirstorder2D}) depends on $a=(a_2, a_3)$
and has the form:
\be\label{DG23}
DG_{2-3}= \bp 0 & -\mbox{Id}_2\\
-M_{2-3}& 0 \ep,
\qquad
M_{2-3}= \mbox{diag}\Big(|a|^2, |a|^2-1\Big) + 2a\otimes a,
\ee
and we find that $DG_{2-3}$ has two couples of eigenvalues
$\{\pm\sqrt{m_j}\}_{j=2,3}$ with corresponding eigenvectors
$(r_j, ~\mp\sqrt{m_j}r_j)$, where:
$m_2 = \frac{1}{2}\left(4|a|^2 - 1 - \sqrt{(2|a|^2-1)^2 + 8a_2^2}\right),$
$m_3 = \frac{1}{2}\left(4|a|^2 - 1 + \sqrt{(2|a|^2-1)^2 +8a_2^2}\right),$
while $r_j = (-2a_2a_3, ~3|a|^2 - 2a_3^2 - m_j)^T$ (or $r_2=(1,0)^T$
when $a_2=0$).
We see that the in the vicinity of $(0,1, b_2, b_3)$ the matrix
$DG_{2-3}$ is strictly hyperbolic,  the two eigenfields
corresponding to $\pm\sqrt{m_3}$ are genuinely nonlinear.


\subsubsection{The 1D incompressible case}\label{334}
For system (\ref{shearfirstorder1D})
the characteristic speeds are $\pm \sqrt{1+3a_1^2}$, while 
for the system (\ref{simplebfirstorder1D}) they are $\pm \sqrt{3a_3^2-1}$.
Hence the second model is strictly hyperbolic for $|a_3|>1/\sqrt{3}$
and elliptic otherwise; 
this can be recognized as agreeing with certain phase-transitional
viscoelasticity models, except that 
the region $a_3\le 0$ (where $\det F\le0$) is unphysical.

\section{Nonlinear stability framework}\label{nstab}

We now briefly recall the general stability theory of
\cite{Z4,R,RZ}, which reduces the question of nonlinear stability
in (\ref{hyppar}) to verification of an Evans function condition.
Namely, given two endstates $V_-$ and $V_+$ belonging to the regions
of strict hyperbolicity of $DG$, we make a smooth change of coordinates $V\mapsto S(V)$ with
$S :\mathbb{R}^{6}\longrightarrow\mathbb{R}^{6}$ given by: 
$S(V)=D\eta(V) = D_aW(a) \oplus b$.
The system (\ref{hyppar}) is equivalent to:
$$\tilde A^0(S) S_t + \tilde A(S)S_z = (\tilde B(S) S_z)_z,$$
where with a slight abuse of notation we use $S=S\circ
V:[0,\infty)\times \mathbb{R}^3\longrightarrow \mathbb{R}^{12}$.
Above:
\begin{equation*}
\begin{split}
\tilde A &= DG(V)\tilde A^0 = 
\left[\begin{array}{cc}0&-\mbox{Id}_3\\-MQ&0\end{array}\right],
\qquad \tilde B = B(V) \tilde A^0 = B(V), \qquad
\tilde A^0=\left[\begin{array}{cc} Q& 0 \\ 0 & \mbox{Id}_3\end{array}\right],
\end{split}
\end{equation*}
where $Q=Q(V)$ is defined as follows. In some open neighborhoods of $V_-$ and $V_+$ 
(where $M$ is positive definite) we set $Q=M^{-1}$, in which case:
\begin{equation*}
\tilde A^0 = \frac{\partial V}{\partial S}= (D^2_V\eta)^{-1} = 
\left[\begin{array}{cc}M^{-1}&0\\0&\mbox{Id}_3\end{array}\right].
\end{equation*}
In the region where $M$ is negative definite, we set $Q=\mbox{Id}_3$.
In between the two above mentioned regions, $Q$ is a smooth,
symmetric and positive definite interpolation of the two matrix fields
$M$ and $\mbox{Id}_3$.
This construction allows us to treat also the case of profiles passing through elliptic
regions, but with hyperbolic endstates (in a similar spirit as for
the van der Waals gas dynamics examples mentioned in \cite{MaZ4,Z4}).

\smallskip

We first check the validity of the structural conditions (A1)--(A3) of \cite{Z4}:

\smallskip

(A1)  $\tilde A(V_-)$ and $\tilde A(V_+)$  are symmetric matrices.
$\tilde A_0$ is symmetric and positive definite (on the whole
$\mathbb{R}^6$). Also, the $3\times 3$ principal minor of $\tilde A$, 
corresponding to the purely hyperbolic part of the system (\ref{hyppar}), 
equals identically $0_3$ hence it is always symmetric, as required.

\smallskip

(A2)  At the endstates $V_\pm$ there holds: 
no eigenvector of $DG$ belongs to the kernel of $B$. In
the region of strict hyperbolicity of $DG$ this condition is equivalent to: no
eigenvector of $M$ is in the kernel of $B_{0,1}$, readily satisfied.

\smallskip

(A3)  $\tilde B$ has the required block structure
$\tilde B=\left[\begin{array}{cc} 0_3&0_3\\0_3&B_0\end{array}\right]$
as in (\ref{blah}). The symmetrization of the minor corresponding 
to the parabolic part of the system (\ref{hyppar}): 
$\mbox{sym } B_{0,i} = B_{0,i}$ is uniformly elliptic in any region in
$V$ of the form: $0<a_3<C$ in case of $B_{0,2}$, and 
$a_3^2 > c(1+a_1^2 + a_2^2)$ in case of $B_{0,1}$ (where $c, C>0$ are
some uniform constants).

\smallskip

We hence find that shock profiles of each
of the planar systems considered in this paper satisfy
conditions (A1)--(A3) of \cite{Z4} defining the
class of symmetrizable hyperbolic--parabolic systems and profiles
to which the theory of nonlinear stability of viscous shock profiles
developed in \cite{MaZ2,MaZ3,MaZ4,Z4,R,RZ} applies,
provided:

(i) the endstates $V_\pm$ lie in the region of strict
hyperbolicity of $DG$,

(ii) The profile $\{\bar V (\cdot)\}$ lies in some region where
the chosen  $B_{0,i}$ is uniformly elliptic.

\medskip

We now validate the additional technical conditions (H0)--(H3) of \cite{Z4}.
Note that the remaining conditions (H4)--(H5) are needed only for the
multi-dimensional systems, as they automatically hold for
systems in 1 space dimension.

(H0)  $G,B,S\in\mathcal{C}^5$.

\smallskip

(H1)  the shock speed $s$ under consideration is non-zero (note that $0$
is the only eigenvalue of the $3\times 3$ principal minor of $DG$,
which indeed is $0_3$).  As remarked in section \ref{s:5}, $s\neq0$ 
for any profile with endstates belonging to the strict
hyperbolicity region of $DG$.

\smallskip

(H2)  $s$ is distinct from the eigenvalues of $DG(V_\pm)$.

\smallskip

(H3)  local to $\bar V(\cdot)$, the set of traveling wave solutions to
(\ref{hyppar}) connecting $(V_-, V_+)$ (with thus determined speed
$s$), forms a smooth finite-dimensional submanifold 
$\{\bar V^\delta(\cdot)\}$ of $\mathcal{C}^1(\mathbb{R},
\mathbb{R}^6)$,  parametrized by $\delta\in B(0, r)\subset
\mathbb{R}^\ell$, and $\bar V^0 = \bar V$.

\subsection{The Evans condition}\label{s:stabtheory}

Linearizing the hyperbolic-parabolic system (\ref{hyppar}) about 
its viscous shock solution of \eqref{hyppar}:
$$ V(z,t)=\bar V(z-st), \quad
\lim_{z\to \pm \infty}\bar V(z)=V_\pm, $$
which satisfies: $-s\bar V + (G(\bar V))_z = (B(\bar V)\bar V_z)_z$, 
and further changing to co-moving coordinates $\tilde z=z-st$,
we obtain the equivalent evolution equations:
\be\label{lin}
V_t=\CalL V:=(\CalB V_z)_z-(\mathcal{G}V)_z.
\ee
Here $\mathcal{G}$ and $\CalB$ are the following matrix fields depending on $z$:
$$\mathcal{G}(z)=DG(\bar V(z)) -s\mbox{Id} - DB(\bar V(z))^T\bar V_z(z),
\qquad \CalB(z) = B(\bar V(z)),$$
and converging asymptotically
to values $\mathcal{G}(\pm \infty)=DG(V_\pm)-s\mbox{Id}$
and $\CalB(\pm \infty)=B(V_\pm)$.
Towards investigating stability of (\ref{lin}), one seeks eigenvalues
$\lambda \in\mathbb{C}$ of $\mathcal{L}$, that is solutions to 
the system $\CalL V = \lambda V$ written in its first-order form:
\be\label{firstorderE}
Z'(z,\lambda)=\mathcal{A}(z,\lambda)Z(z,\lambda).
\ee
The augmented ``phase variable'' $Z$ consists of  $V=(a,b)$ and
the derivative $b'$ of its parabolic-like component.

As shown in \cite{GZ,ZH,MaZ3,MaZ4,Z4},
under conditions (H0), (H1), (H2) it is possible to
define an analytic {\it Evans function} 
$D: \{\lambda\in\mathbb{C}; Re ~\lambda \ge 0\}\longrightarrow
\mathbb{C}$  associated with (\ref{firstorderE})
and hence consequently associated with $\CalL$ and with the original problem (\ref{hyppar}).
We shall now briefly sketch this construction, for further details see e.g. \cite{AGJ,GZ,Z4,HuZ}.

In the first step one observes that the complex matrix field 
$\mathcal{A}(z,\lambda)\in\mathbb{C}^{N\times N}$ 
in (\ref{firstorderE}) is analytic in $\lambda$ and has an exponential decay 
to the respective $\mathcal{A}_\pm(\lambda)$ as
$z\to\pm\infty$ (uniformly in bounded $\lambda$). 
The second step consists in proving that (\ref{firstorderE}) on each of
the half-lines $(-\infty,0]$ and $[0,\infty)$, is equivalent to:
$$\tilde Z'(z) = \mathcal{A}_-(\lambda)\tilde Z(z), \quad z\leq 0
\qquad \mbox{ and }\qquad 
\tilde Z'(z) = \mathcal{A}_+(\lambda)\tilde Z(z), \quad z\geq 0,$$
under change of variables $ Z(z) = P_-(z,\lambda)\tilde Z(z)$ for
$z\leq0$, and $ Z(z) = P_+(z,\lambda)\tilde Z(z)$ for $z\geq0$.
Existence of such (non-unique) analytic in $\lambda$ and invertible matrix fields
$P_\pm(z,\lambda)\in\mathbb{C}^{N\times N}$, decaying exponentially to $\mbox{Id}$
as $z\to\pm\infty$, is achieved  by a conjugation lemma \cite{Z4}.

Further, denote by $\{\tilde Z_i^+(\lambda)\}_{i=1..k}$ the (analytic in
$\lambda$) basis of the stable space $\mathcal{S}$ of $\mathcal{A}_+(\lambda)$,
and likewise let $\{\tilde Z_i^-(\lambda)\}_{i=k+1..N}$ be the 
basis of the unstable space $\mathcal{U}$ of $\mathcal{A}_-(\lambda)$, where 
the consistency of the dimensions follows from assumptions (H1), (H2).  
Define:
$$Z_i^+(z,\lambda) = P_+(z,\lambda)\tilde Z_i^+(\lambda), \quad z\geq 0
\qquad \mbox{ and } \qquad Z_i^-(z,\lambda) = P_-(z,\lambda)\tilde
Z_i^-(\lambda), \quad z\leq 0.$$
Clearly, given any $Z_0\in span\{Z_i^+(z_0,\lambda) \}_{i=1..k}$,
$z_0\geq 0$, there exists a solution to (\ref{firstorderE}) on
$[z_0,\infty)$ decaying exponentially to $0$ as
$z\to\infty$, and with initial data $Z(z_0) = Z_0$. 
It has the property that $Z(z,\lambda)\in
 span\{Z_i^+(z,\lambda) \}_{i=1..k}$ for all $z\geq z_0$. 
A similar assertion of backward resolvability of (\ref{firstorderE})
is true for  $Z_0\in span\{Z_i^-(z_0,\lambda)\}_{i=k+1.. N}$,
$z_0\leq 0$ with exponential decay at $z\to -\infty$.

The Evans function is now introduced as the following Wronskian:
\be\label{Ddef1}
D(\lambda)=\det \Big(Z_1^+(0,\lambda), \dots, Z_k^+(0,\lambda),
Z_{k+1}^-(0,\lambda), \dots, Z_N^-(0,\lambda)\Big).
\ee
Away from the origin $\lambda=0$, $D$  vanishes at $\lambda$ with
$Re~\lambda\geq 0$ if and only if $\lambda$
is an eigenvalue of $\CalL$, corresponding to existence of a solution
$Z(z,\lambda)$  of $\mathcal{L}Z = \lambda Z$, 
decaying to $0$ at both $z\to \pm \infty$.
Indeed, the multiplicity of the root is equal to the multiplicity
of the eigenvalue \cite{GJ1,GJ2,MaZ3,Z4}.
The meaning of the multiplicity of the root of $D$ at embedded eigenvalue
$\lambda=0$ is less obvious, but is always greater than or equal to the
order of the embedded eigenvalue \cite{MaZ3,Z4}.

In agreement with \cite{MaZ3,Z4}, we define the {\it Evans stability condition}:
\medskip

(D) $D$ has no root in $\{Re ~\lambda\geq 0\}$ except for $\lambda=0$,
which is the root of multiplicity $\ell$.
\medskip

Note that under assumption (H3), the condition (D) is equivalent to
$D$ having precisely $\ell$ zeros 
in $\{Re ~\lambda\geq 0\}$.

\subsection{Type of the shock}\label{s:type}

Define:
\begin{equation*}
\begin{split}
\tilde \ell  = &~ \mbox{dimension of
the unstable subspace of } DG(V_-) \\
& +  \mbox{ dimension of the stable
subspace of } DG(V_+) - \dim V,
\end{split}
\end{equation*}
where $\dim V = 6$ is the dimension of the whole space.
Then, the hyperbolic shock $(V_-,V_+)$ is
defined to be: 

(i) of {\it Lax type} if $\tilde \ell=1$, 

(ii) of {\it overcompressive type} if $\tilde \ell>1$,

(iii) of {\it undercompressive type} if $\tilde \ell <1$.


\noindent If $\tilde \ell=\ell \ge 1$ or $\tilde \ell < \ell=1$, with
$\ell$ as in (H3),
then the viscous shock $\bar V$ is defined to be of {\it pure}
{\it Lax}, {\it overcompressive}, or {\it undercompressive} type, according to
the hyperbolic classification just above.
Otherwise, $\bar V$ is defined as {\it mixed under-overcompressive type},
\cite{LZu,ZH,MaZ3,Z4}.  All the shocks considered
in this paper appear to be of pure type.  Indeed, though artificial
examples are easily constructed \cite{LZu,ZH},
we do not know of any physical example of a mixed-type shock.

\subsection{Linear and nonlinear stability}\label{s:stabsec}

Consider a planar viscoelastic shock for which the endstates $V_\pm$ 
lie in the region of strict hyperbolicity and profile $\{\bar V(\cdot)\}$
lies in the region for which $B_{0,i}$ is uniformly elliptic.

We have the following basic results relating the Evans condition
(D) to stability.

\begin{proposition}[\cite{MaZ3}]\label{Dcond}
Assume (H0), (H2) and (H3). The Evans condition (D) is necessary and sufficient for
the linearized stability $L^1\cap L^p \to L^p$ of $\bar V$, for all
$1\le p\le \infty$: 
$$ \|e^{t\mathcal{L}} f\|_{L^p}\leq C\left(\|f\|_{L^1} + \|f\|_{L^p}\right).  $$
\end{proposition}

\begin{proposition}[\cite{MaZ4,RZ}]\label{orbital}
Assume (H0), (H2), (H3) and (D). Then we have:

{\rm{(i) Stability.}} For any initial data $\tilde V(\cdot, 0)$ with:
$$E_0:=\|(1+|z|^2)^{3/4}( \tilde V(\cdot, 0)- \bar V)\|_{ H^5} <<1$$
sufficiently small, a solution $\tilde V$ of \eqref{hyppar} exists for
all $t\ge 0$ and:
\begin{equation} \label{stabstatement}
\|(1+|z|^2)^{3/4}( \tilde V(\cdot, t)- \bar V(\cdot-st))\|_{ H^5}\le CE_0.
\end{equation}

{\rm{(ii) Phase-asymptotic orbital stability.}}
There exist $\alpha(t)$ and $\alpha_\infty$ such that:
\begin{equation}\label{stabstatement2}
\| \tilde V(\cdot, t) -  \bar V^{\alpha(t)}(\cdot -st)\|_{L^p}
\le CE_0(1+t)^{-(1-1/p)/2}
\end{equation}
and:
\begin{equation}
\label{phasebd}
| \alpha(t)- \alpha_\infty|\le C E_0 (1+t)^{-1/2}, 
\qquad |\dot \alpha(t)| \le C E_0 (1+t)^{-1},
\end{equation}
for all $1\le p\le \infty$.
\end{proposition}

\begin{lemma}\label{DH}
Assume (H0), (H2) and (D). If $\tilde \ell=\ell$ or $\ell=1$,
with $\ell$ as in (D), then  (H3) holds with the same value $\ell$.
In particular, these conditions together
imply nonlinear time-asymptotic orbital stability.
\end{lemma}

\begin{proof}
The claim follows by the existence theory of \cite{MaZ3},
relating the dimensions of stable and unstable manifolds of
the rest points $V_\pm$ in the traveling-wave ODE, to the hyperbolic
index $\tilde \ell$.
Further \cite{GZ,ZH,MaZ3},
stability condition (D) implies ``maximal transversality'' consistent with existence of
a profile of the traveling-wave connection as a solution of the
traveling-wave ODE  (i.e. actual transversality), yielding 
(H3) with $\tilde \ell=\ell$ in the case $\tilde\ell\ge 1$,
and (H3) with $\ell=1$.
\end{proof}

Combining Proposition \ref{orbital} with Lemma \ref{DH}, we obtain:

\bt\label{main2}
For each of the planar systems considered in this
paper, every viscous Lax, overcompressive, or undercompressive shock
satisfying:
\begin{itemize}
\item[\rm{(i)}] condition (H2) (noncharacteristicity),
\item[\rm{(ii)}] with endstates lying in the region of strict hyperbolicity of $DG$,
\item[\rm{(iii)}] with profile lying in the region of uniform ellipticity of $B_{0,i}$,
\item[\rm{(iv)}] satisfying (D), 
\end{itemize}
is linearly and nonlinearly orbitally stable.
\et

In particular, Propositions \ref{Dcond}, \ref{orbital} and Theorem \ref{main2} apply to 
profiles with $a\in\mathbb{R}^3$ such that the corresponding $F$ of
the form (\ref{Fplanar}) is contained in a sufficiently
small neighborhood of $SO(3)$.  In the incompressible shear case, 
they apply to any profile with endstates $a_\pm\ne 0$.

The condition (H2) corresponds to noncharacteristicity of the shock,
which holds generically. It guarantees also exponential decay of the shock to
its endstates \cite{MaZ3,Z4},
which is needed for efficient numerical approximation of the profile.

Finally, we remark that strict hyperbolicity at $V_\pm$
is not necessary for existence of profiles, 
but only to apply the basic stability framework
developed in this section.
When hyperbolicity fails, the corresponding endstate
is unstable as a constant solution; however, this
instability can be stabilized by convective effects
if unstable modes are convected sufficiently rapidly
into the shock zone; see Appendix \ref{s:complex}.
This situation cannot occur for shear flows, for which
all states are hyperbolic, but would be interesting to
investigate in the compressible case.

\subsection{The integrated Evans condition}\label{s:Evansint}
Making the substitution $\tilde V(z) = \int_{-\infty}^z
V(y)~\mbox{d}y$ and integrating the equations in  
(\ref{firstorderE}) from $-\infty$ to $z$, we obtain after dropping
the tilde notation:
\begin{equation}\label{inteval1}
\lambda V = \tilde{\mathcal{L}} V:= \mathcal{B} V'' - \mathcal{G}V'.
\end{equation}
We conclude for any $\lambda \ne 0$, 
satisfaction of (\ref{firstorderE}) for a solution $V$ decaying
exponentially up to one derivative, implies that
$\tilde V(z)$ is also exponentially
decaying and satisfies (\ref{inteval1}).

Associated with $\tilde \CalL$ is an {\it integrated Evans function}
$\tilde D(\lambda)$, which like $D$ is analytically defined
on the nonstable half-plane $\{Re ~\lambda \ge 0\}$,
through the construction sketched in section \ref{s:stabtheory}.
In the Lax and overcompressive cases,
the change to integrated coordinates has the
effect of removing the zeros of $D$ at the origin, making the Evans
function easier to compute numerically and hence the stability 
condition easier to verify.

\begin{proposition}[\cite{ZH,MaZ3}]\label{intcond}
Assume (H0), (H2). Then the Evans condition (D) is equivalent to the 
following {\rm integrated Evans condition}:
\begin{itemize}
\item[\rm{(i)}] for the Lax and  overcompressive shock types:

\smallskip

$\rm (\tilde D)~~$ the integrated Evans function $\tilde D$ is
nonvanishing on $\{ Re ~\lambda \ge 0\}$,
\medskip

\item[\rm{(ii)}] for the undercompressive shock type:
\medskip

$\rm (\tilde D')~~$ the function  $\tilde D$ has on 
$\{Re~ \lambda \ge 0\}$
a single zero of multiplicity $1+|\tilde \ell|$, at $\lambda=0$.
\end{itemize}
\end{proposition}
Note that the inclusion of term $|\tilde \ell|$ repairs an omission in
\cite{HLZ}, for which  $\tilde \ell\equiv 0$ in the undercompressive case.
Propositions \ref{intcond}, \ref{main2},
\ref{orbital} and \ref{Dcond}
give together a simple and readily numerically
evaluated test for stability of large-amplitude and or non-Lax-type
waves.

\subsection{Small-amplitude stability}\label{s:smallamp}
The following proposition gives a first 
nonlinear stability result for planar viscoelastic shocks,
answering a conjecture posed in \cite{AM} for the shear wave case.

\begin{proposition}[\cite{HuZ}]\label{smallstab}
Assume (H0). Let $V_0$ be a point of strict hyperbolicity of $DG$ and let
$\lambda_0$ be one of its eigenvalues, associated with a genuinely
nonlinear characteristic field. Then there exists $\epsilon>0$
sufficiently small such that for any viscous shock $\bar V$ with speed
$s$ satisfying:
$$\|\bar V-V_0\|_{L^\infty}<\epsilon \quad
\mbox{ and } \quad |s-\lambda_0|<\epsilon,$$
we have:

{\rm{(i)}} the shock is of Lax type, 

{\rm{(ii)}} the Evans condition (D) holds, hence $\bar V$  is
linearly and nonlinearly phase-asympto\-ti\-cal\-ly orbitally stable.
\end{proposition}


\section{Existence of viscous shock profiles}\label{s:5}
Let us now seek traveling waves connecting given endstates:
$$V_- = V(-\infty) = (\alpha,0), \qquad V_+=V(+\infty)=(a_+, b_+).$$
Indeed, by invariance of (\ref{visco_eq}) under change in coordinate
frame $\xi\mapsto \xi + b_0 t$, we may without loss of generality
assume that $b(-\infty)=0$.

Hereafter we restrict to the simpler (and apparently more physical)
case of viscosity tensor $\mathcal{Z}_2$.
The case $\mathcal{Z}_1$ may be treated similarly.
We note that the type and location of equilibria of the traveling wave
ODE under are assumptions are independent of the choice of $\mathcal{Z}$,
by the general results of \cite{MaZ3}; see \cite{BLZ} for further
discussion in the somewhat similar context of MHD.

\medskip

Writing the profile equation for \eqref{hyppar} with (\ref{part}) and
(\ref{Zs}),  we obtain:
\begin{equation}\label{5.0}
-sa'-b'=0,\qquad
-sb'-DW_0(a)'=\Big(\frac{(b_1', b_2', 2b_3')}{a_3}\Big)'.
\end{equation}
Note that $s\neq 0$ for profiles satisfying the nonlinear stability
conditions (namely, the endstates belonging to the strict
hyperbolicity region of $DG$). For otherwise $b'=0$ and $DW_0(a)'=0$
hence $M(a) a'=0$ along the profile, contradicting the invertibility
of $M$ in the neighborhood of $a(-\infty)$.

Now, substituting the first equation into the second, making the
change of variable $z\mapsto sz$, and
defining $\sigma= s^2$, we get the following {\it reduced profile equation}:
\ba\label{redprof}
-\sigma a'+ DW_0(a)'&=\Big(\frac{(a_1', a_2', 2a_3')}{a_3}\Big)',\\
\ea
recognized as associated with the
strictly parabolic gradient flux system in $a$ alone:
\begin{equation}\label{sys}
a_t+ DW_0(a)_z=\Big(\frac{(a_1, a_2, 2a_3)_z}{a_3}\Big)_z.
\end{equation}
Note that $\eta(a)=\frac{|a|^2}{2}$ is the convex entropy for
(\ref{redprof}) as: 
$$\nabla\eta(a) \cdot D^2_a W(a) = \nabla q(a), \qquad q(a) = a\cdot
D_aW(a) - W(a).$$

Evidently, \eqref{redprof} may be written as a generalized gradient flow:
\be\label{gradflowU}
\frac{(a_1', a_2', 2a_3')}{a_3}=
\nabla_a \phi(a),
\qquad
\phi(a)=W_0(a)-\sigma \frac{|a|^2}{2} -(DW_0(\alpha)-\sigma
\alpha)\cdot a,
\ee
where $\alpha=a(-\infty)$. Making the change of variable
$z\mapsto\tilde z(z)$ where $\tilde z$ solves the ODE: 
$\tilde z'(z) = 1/a_3(\tilde z(z))$, the
system (\ref{gradflowU}) becomes:
\begin{equation}\label{5.45}
(a_1', a_2', 2a_3') = \nabla_a \phi(a).
\end{equation}

We see that the function $z\mapsto \phi(a(z))$ is non-decreasing:
\begin{equation}\label{increase}
(\phi\circ a)' = (\nabla\phi)a' = a_3~
\mbox{diag}\left(1,1,\frac{1}{2}\right)
\nabla\phi(a)\otimes\nabla\phi\geq 0
\end{equation}
in the admissible region  $a_3>0$.
This is a simple instance of a more general fact
concerning parabolic conservation laws
possessing a viscosity-compatible strictly convex entropy \cite{G,CS1,CS2,BLZ}. 
Moreover, the type of the shock connection of the original viscoelasticity
equations is the same as the type for the reduced equations \eqref{sys},
which is in turn determined by the relative Morse index of the 
endstates/equilibria considered as critical points $a$:
\begin{equation*}\label{critpt}
DW_0(a)-\sigma a- (DW_0(\alpha)-\sigma \alpha)=0
\end{equation*}
of $\phi$. See also the general results and discussion of
\cite{MaZ3,BLZ}.

\smallskip

Finally, a straightforward calculation shows that:
\begin{equation}\label{h1}
\begin{split}
s\phi(a) = &s\eta(V) - \big(q(V) + \zeta\big) + \nabla q(V) \cdot
\Big(G(V) - G(V-) - s(V - V_-)\Big)\\
&\mbox{for } V=(a,b) \mbox{ with } b=-s(a-\alpha),
\end{split}
\end{equation}
where the inviscid flux $G$, entropy $\eta$ and entropy flux $q$ are
as in (\ref{dg}) and (\ref{en_flu}).
The relation $b=-s(a-\alpha)$ is valid along the profile, and it
follows by integrating the first equation in (\ref{5.0}) from
$-\infty$ to $z$. The vector $\zeta=sDW(\alpha)\alpha - \frac{1}{2}s^3|\alpha|^2$,
which is independent of $V$, can be seen as an adjustment of the
entropy flux $q$, naturally defined up to a constant.

The quantity in the right hand side of (\ref{h1}) is related to the
dissipative quantity:
$$\psi(V) = -s\eta(V) + q(V)$$
which decreases across any viscous profile connection lying within the
region of strict hyperbolicity of the reference hyperbolic-parabolic
system and the region of strict convexity of its entropy $\eta$
(see\cite{BLZ} and references therein):
$$\psi(V_+) - \psi(V_-) < 0.$$
Indeed, by (\ref{h1}) and the Rankine-Hugoniot relations, it follows
that:
$$\psi(V_-) = -s\phi(\alpha) ~~~~\mbox{ and } ~~~~\psi(V_+) = -s\phi(a_+).$$
Thus, in view of (\ref{increase}), we conclude that in the present
setting $\psi$ is decreases across any viscous profile with positive
speed $s>0$, even one passing the elliptic region. This clarifies
somewhat the role of $\phi$ in the original system.

\subsection{The 3D compressible system}

Recalling (\ref{w0}), (\ref{5.45}) becomes:
\ba\label{compode}
{a_1'}&= (|a|^2-\sigma)  a_1 -(|\alpha|^2-\sigma)  \alpha_1 ,\\
{a_2'}&= (|a|^2-\sigma)  a_2 -(|\alpha|^2-\sigma)  \alpha_2 ,\\
{2 a_3'}&= (|a|^2-1-\sigma)a_3- (|\alpha|^2-1-\sigma)\alpha_3.
\ea
As $|a|\to \infty$, $\phi(a)\sim \frac{|a|^4}{4}$, hence the phase
portrait of (\ref{gradflowU}) always possesses a minimum, or repellor.
More, $\nabla \phi(a)\sim |a|^2 a$ points in the outward radial
direction, and hence the index of this vector field
on a suitably large ball is $+1$, and it must be equal to the sum of the indices
of the equilibria (generically five - see Section \ref{s:port3D} and
Figure \ref{compressible_phase_portrait_1}), defined as the
signs of the associated Jacobians
$\sgn \det(D^2W_0-\sigma\Id).$
The same argument shows that a sufficiently large
ball is absorbing in backwards $z$, so that we can conclude
that any orbit lying in the  stable manifold of an equilibrium 
must connect in backward $z$ to some other equilibrium possessing
an unstable manifold.

Further, when $a_3=0$ we have $\partial_3\phi = -
(|\alpha|^2-1-s^2)\alpha_3$, which is independent of $(a_1,a_2)$.
Since $\nabla\phi\sim |a|^2a$ as $|a|\to +\infty$, it follows that the
index of $\nabla\phi$ on a large half-ball: $B_R(0)\cap\{a_3>0\}$
equals $+1$ for $(|\alpha|^2-1-\sigma)\alpha_3>0$,  and it equals $0$ for
$(|\alpha|^2-1-\sigma)\alpha_3<0$.
In the former case, the region $B_R(0)\cap\{a_3>0\}$ is invariant in backward
$z$ and so we may conclude that
any orbit lying in the  stable manifold of an equilibrium 
in $\{a_3>0\}$ must connect in backward $z$ to some other equilibrium 
in $\{a_3>0\}$ possessing an unstable manifold. 


\subsection{The 2D incompressible shear case}
The incompressible case can be analyzed similarly as above, with (\ref{sys}) becoming: 
$$ a_t+ DW_0(a)_z=a_{zz},$$
where $a=(a_1, a_2)$ and  $W_0(a)=\frac{1}{4}|a|^4 + \frac{1}{2}|a|^2$.
That gives a $2\times2$ rotationally symmetric model:
\ba\label{shearfirstorderprof}
a_{t} +( (|a|^2+1)a)_z= a_{zz},
\ea
of a form $a_t + (h(|a|)a)_z=a_{zz}$ that has been
much studied as a prototypical example of a system
with rotational degeneracy \cite{F}.
Further, the counterpart of (\ref{gradflowU}) reads:
$$ 
a'= \nabla_a \phi(a),
\qquad
\phi(a)=W_0(a)-\sigma \frac{|a|^2}{2}
-(DW_0(\alpha)-\sigma \alpha)\cdot a. $$
Expanded in coordinate form, the profile ODE reads:
\begin{equation*}\label{shearode}
\begin{split}
a_1'&= (|a|^2+1 -\sigma)  a_1 -(|\alpha|^2+1 -\sigma)  \alpha_1 ,\\
a_2'&= (|a|^2+1 -\sigma)  a_2 -(|\alpha|^2+1 -\sigma)  \alpha_2 .\\
\end{split}
\end{equation*}
Again, we find by an asymptotic development of $\phi$ that
the vector index of this ODE on a suitably large ball is $+1$,
and there exists always at least one repellor, with index $+1$.
In the generic case (see below), there are three nondegenerate
equilibria, each of index $\pm 1$: one is of index $+1$ 
and of one index $-1$.




Writing $\alpha=a(-\infty)$ and $a_+=a(+\infty)$, 
the Rankine--Hugoniot relations for (\ref{shearfirstorderprof}) are:
\be\label{rh}
(|a_+|^2+1-\sigma)a_+ - (|\alpha|^2 +1 -\sigma)\alpha = 0.
\ee
By rotation invariance, we may restrict our attention to $\alpha=(\alpha_1,0)$.
We shall distinguish two cases.

\medskip

{\bf Case (i) $\mathbf{\alpha_1= 0}$.}  We find that
there is a circle of solutions $a_+$ to (\ref{rh}), given by $|a_+|^2=\sigma-1$,
surrounding the rest state $a_+=0$ at the center.
Along each radius of the circle, there is a viscous shock connection
$a(t,z) = \rho(z-\sigma t)e^{i\theta}$ solution to (\ref{shearfirstorderprof})
whose norm $\rho=|a|$ satisfies:
\be\label{rho}
\rho_t + (\rho^3+\rho)_z=\rho_{zz}
\ee
and connects $\rho_+= \rho(+\infty) = \sqrt{\sigma -1}$ 
to $\rho_-=\rho(-\infty) = 0$.
Note that (\ref{rho}) is also the associated parabolic equation to the
flow:
\begin{equation*}
a_1'=(a_1^2+1 -\sigma)a_1-(\alpha^2+1-\sigma)\alpha,
\end{equation*}
which is the counterpart of (\ref{gradflowU}) for the 1d
incompressible model (\ref{shearfirstorder1D}).
When $\alpha_1=a_1(-\infty) = 0$ then $a_{1+}=a_1(+\infty)=\sqrt{\sigma
  -1}$ and thus we obtain:
$$ a_1'=(a_{1}^2 - \alpha_{1}^2)a_1,$$
which has the explicit solution:
\begin{equation}\label{explicitsoln}
a_1(z)=\frac{\alpha_1 \exp(-\alpha_1^2 z)}{\sqrt{k+\exp(-2\alpha_1^2z)}}
\end{equation}
with $k>0$.
Note that this solution connects $a(-\infty)= \sqrt{\sigma -1}$ 
to $a(+\infty)=0$; that is, the connection goes in opposite direction from
the one sought.
  Setting now $a(t,z) = a_1(z-\sigma t)e^{i\theta}$ (with
constant rotation angle $\theta$) gives the traveling viscous shock
solution to (\ref{shearfirstorderprof}).

\br\label{circlermk}
Though noncharacteristic when considered as one-dimensional solutions,
as reflected by uniform exponential convergence to their endstates
(see discussion below Theorem \ref{main2}),
such shocks are always characteristic with respect to the transverse
(rotational) modes, which have characteristic speeds $\pm \sqrt{|a|^2+1}$
equal to $\pm \sqrt \sigma=\pm s$.
\er

\medskip

{\bf Case (ii) $\mathbf{\alpha_1\neq 0}$.}
When $a_{2+}=0$ then (\ref{rh}) reduces to $(a_{1+}^2 +
1-\sigma)a_{1+} = (\alpha_1^2 +1-\sigma)\alpha_1$ corresponding to the
associated scalar equation (\ref{rho}). 

When $a_{2+}\neq0$ then the second equation in (\ref{rh}) becomes
$a_{1+}^2 + a_{2+}^2 = |a_+|^2 = \sigma-1$ whence, from the first
equation: $\alpha_1^2 = |\alpha|^2 = \sigma-1$. That is, solutions
with $a_{2+}\neq0$ exist only if $\sigma$ is equal to the linearly
degenerate characteristic speed, which has no profile.

Thus we may without loss of generality restrict to the (at most)
triples of possible rest states $(\alpha_1^{(i)},0)$ with (to fix the
ideas):
$$ \alpha_1^{(3)} < 0 < \alpha_1^{(2)} < \alpha_1^{(1)}$$
and $((\alpha_1^{(i)})^2 +1-\sigma) \alpha_1^{(i)}<0$ sufficiently
small. Considering the equation (\ref{rho}), we find that the
outermost rest points  $(\alpha_1^{(3)},0)$ and  $(\alpha_1^{(1)},0)$
are connected to the innermost  $(\alpha_1^{(2)},0)$ by a scalar (1d)
shock profile. Indeed:
\begin{itemize}
\item $ (\alpha_1^{(1)})^2 + 1-\sigma <0< 3 (\alpha_1^{(1)})^2+1-\sigma$,
so $(\alpha_1^{(1)},0)$ is a saddle,
\item $(\alpha_1^{(2)})^2 + 1-\sigma < 3 (\alpha_1^{(2)})^2+1-\sigma <0$,
so $(\alpha_1^{(2)},0)$ is an attractor,
\item $ 3(\alpha_1^{(3)})^2 + 1-\sigma > (\alpha_1^{(3)})^2+1-\sigma>0$,
so $(\alpha_1^{(3)},0)$ is a repellor.
\end{itemize}
The phase portrait thus consists of a family of overcompressive
profiles connecting $(\alpha_1^{(3)},0)$ and $(\alpha_1^{(2)},0)$,
bounded by Lax shocks between $(\alpha_1^{(3)},0)$ and
$(\alpha_1^{(1)},0)$, and between $(\alpha_1^{(1)},0)$ and $(\alpha_1^{(2)},0)$,
similarly as for the closely
related ``cubic model'' studied, e.g., in \cite{F,Br}.
See Figure \ref{shear_phase_portrait_1} for typical phase
portraits computed numerically using MATLAB.

\begin{center}
\begin{figure}[htbp] 
$\begin{array}{lr}
(a) \includegraphics[scale=.4]{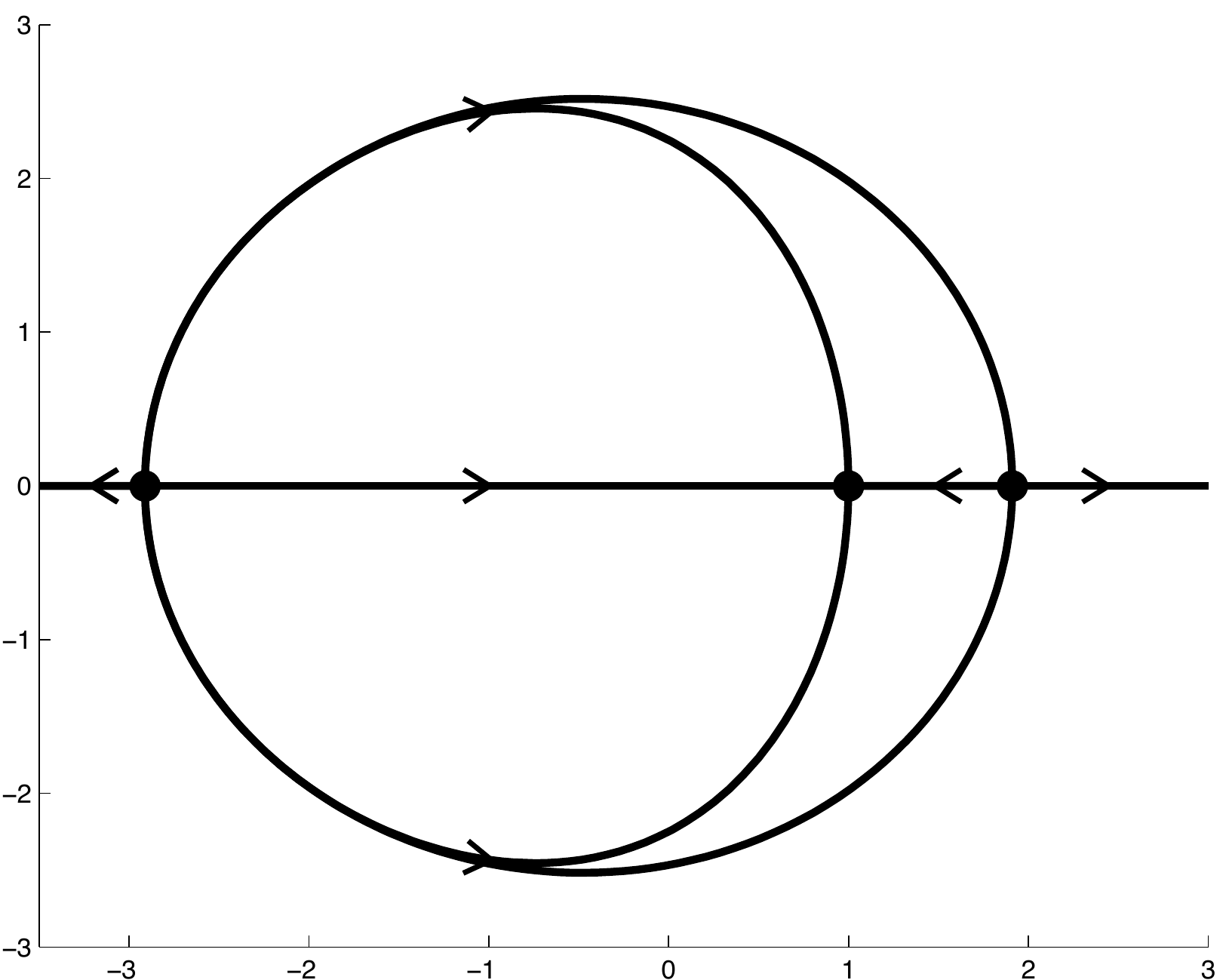}&(b) \includegraphics[scale=.4]{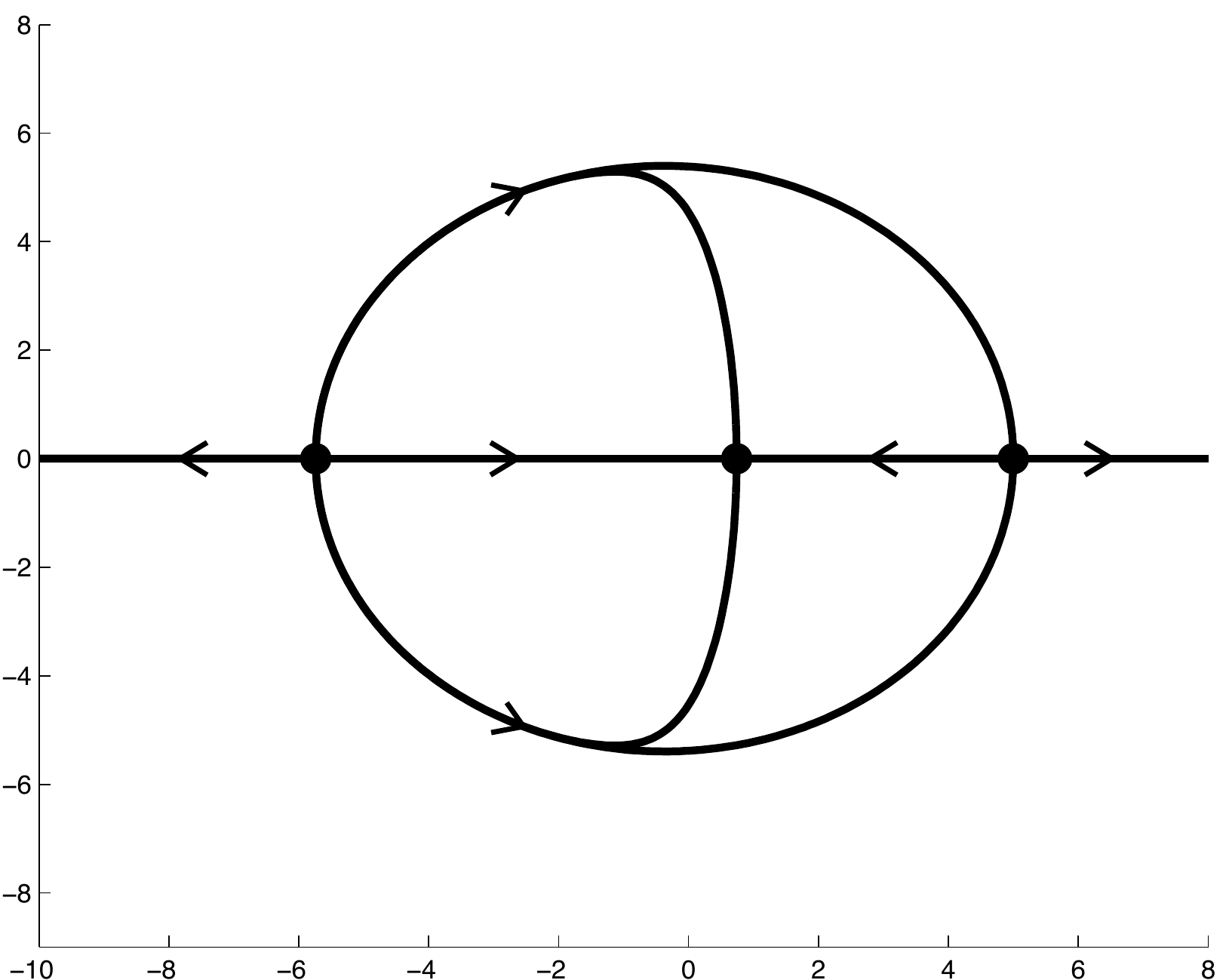}
 \end{array}$
 \caption{Typical phase portrait for the 2D shear case. In (a) we
   have $\alpha_1=1$ and $\sigma=11/4$, and in (b) we have $\alpha_1=5$ and $\sigma=121/4$.} \label{shear_phase_portrait_1}
\end{figure}
\end{center}

\subsection{The 2D compressible case}\label{s:port2D}
Here (\ref{sys}) simplifies to:
\ba\label{eqbfirstorder2}
a_{2,t} +  (|a|^2 a_2)_z &= \left(\frac{a_{2,z}}{a_3}\right)_z,\\
a_{3,t} +  ((|a|^2-1)a_3)_z &= 2\left(\frac{a_{3,z}}{a_3}\right)_z,
\ea
where $a=(a_2, a_3)$ and 
Rankine--Hugoniot relations are:
\ba\label{rhcomp2D}
(|a_+|^2-\sigma) a_{2+} -  (|\alpha|^2-\sigma) \alpha_2 &= 0\\
(|a_+|^2-1-\sigma) a_{3+} -  (|\alpha|^2-1-\sigma) \alpha_3 &= 0,
\ea
where $\alpha=(\alpha_2, \alpha_3)=a(-\infty)$, $a_+=(a_{2+}, a_{3+})=a(+\infty)$
and $a_{3+}, \alpha_3>0$.
We shall distinguish two cases.

\medskip

{\bf Case (i) $\mathbf{\alpha_2=0}$.} 
Strict hyperbolicity of (\ref{DG23}) enforces that
$\alpha_3>1/\sqrt{3}$ and $\alpha_3\neq 1/\sqrt{2}$.
When $a_{2+}=0$ then (\ref{rhcomp2D}) implies that:
$$a_{3+}= -\frac{1}{2}\alpha_3\pm\frac{1}{2}\sqrt{4(1+\sigma) - 3\alpha_3^2},$$
with at most one physically feasible solution $a_{3+}>0$.
For every $\alpha_3$ the range of $\sigma$, for which $a_{3+}>1/\sqrt{3}$ 
and $a_{3+}\neq 1/\sqrt{2}$ is:
$$\sigma \in (\alpha_3^2 + \frac{1}{\sqrt{3}}\alpha_3 - \frac{2}{3},
+\infty) \setminus\{\alpha_3^2 + \frac{1}{\sqrt{2}}\alpha_3 - \frac{1}{2}\}.$$
An associated 1d traveling wave of the type $(0,a_3(z))$ must satisfy:
$$-\sigma a_3' + (a_3^3 - a_3)' = 2\left(\frac{a_3'}{a_3}\right)'.$$

If $a_{2+}\neq 0$, then the first equation in (\ref{rhcomp2D}) implies
$|a_+|^2=\sigma$, while by the second equation:
$a_{3+}=  (1+\sigma-\alpha_3^2 )\alpha_3$, hence:
$$a_{2+}=\pm\sqrt{\sigma -(1+\sigma-\alpha_3^2)^2\alpha_3^2}.$$
We see that there is a pattern of at most four physically feasible equilibria, 
corresponding to two 1d solutions plus two more symmetrically 
disposed about the $a_3$ axis,
There are at most five equilibria in total, counting a fifth possible
infeasible radial solution with $a_{3}<0$.
Here, we are ignoring the line of nonphysical equilibria $a_3=0$
induced by the form of the viscosity tensor.
See Figures \ref{compressible_phase_portrait_1} (a) (c)
for a typical phase portrait computed numerically using MATLAB.

\begin{center}
\begin{figure}[htbp] 
$\begin{array}{lr}
 (a) \includegraphics[scale=.4]{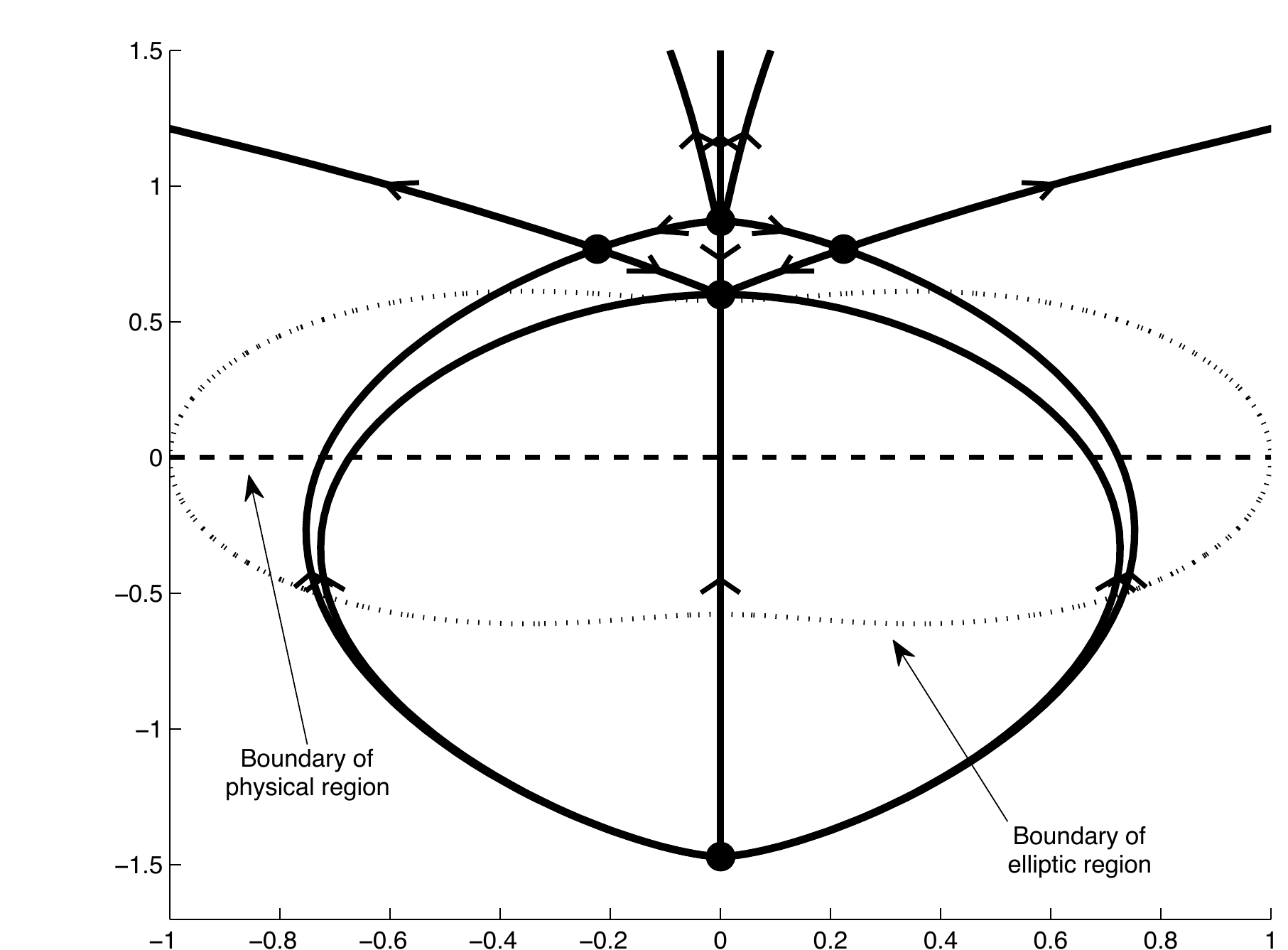}& (b) \includegraphics[scale=.4]{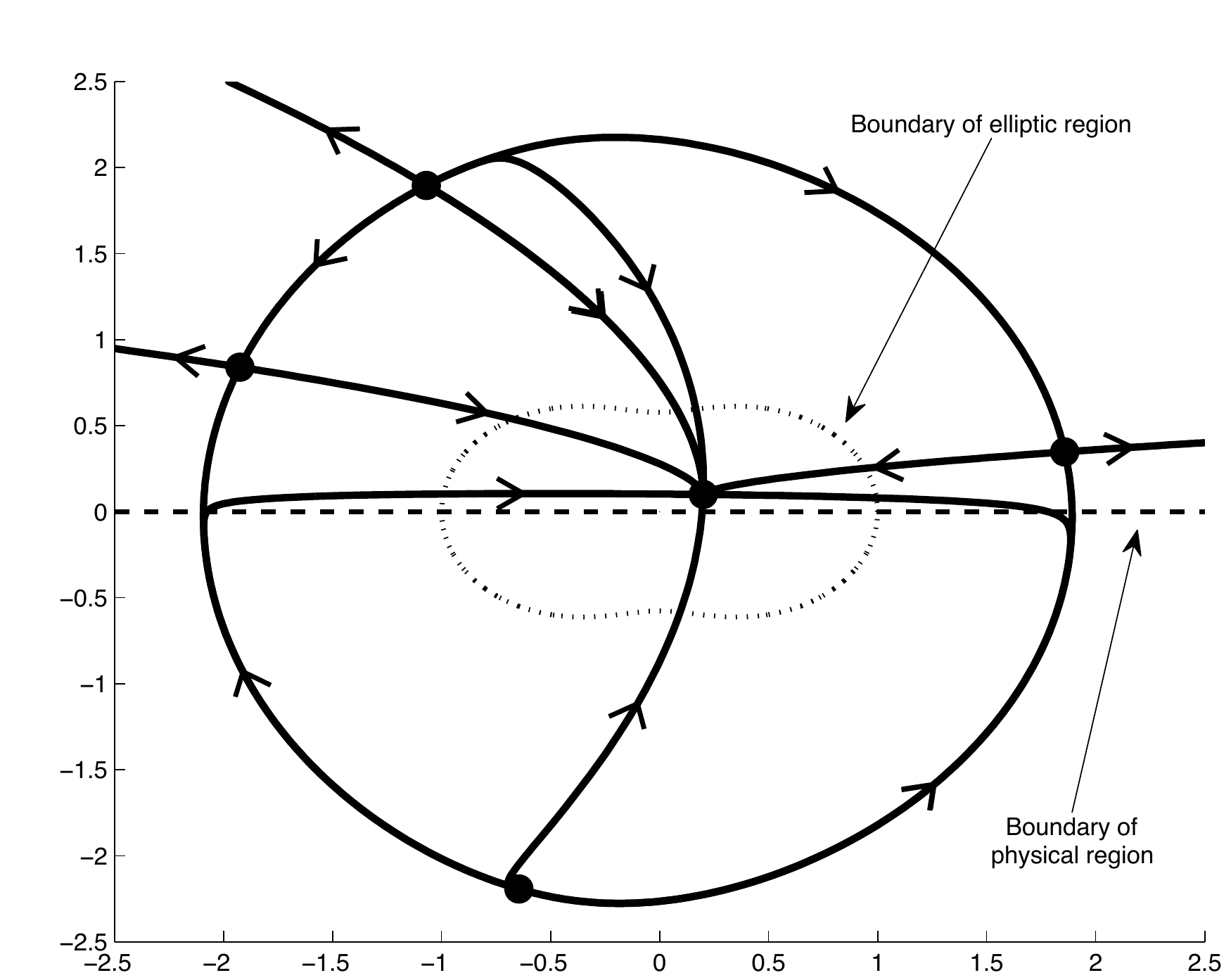}\\
 (c) \includegraphics[scale=.4]{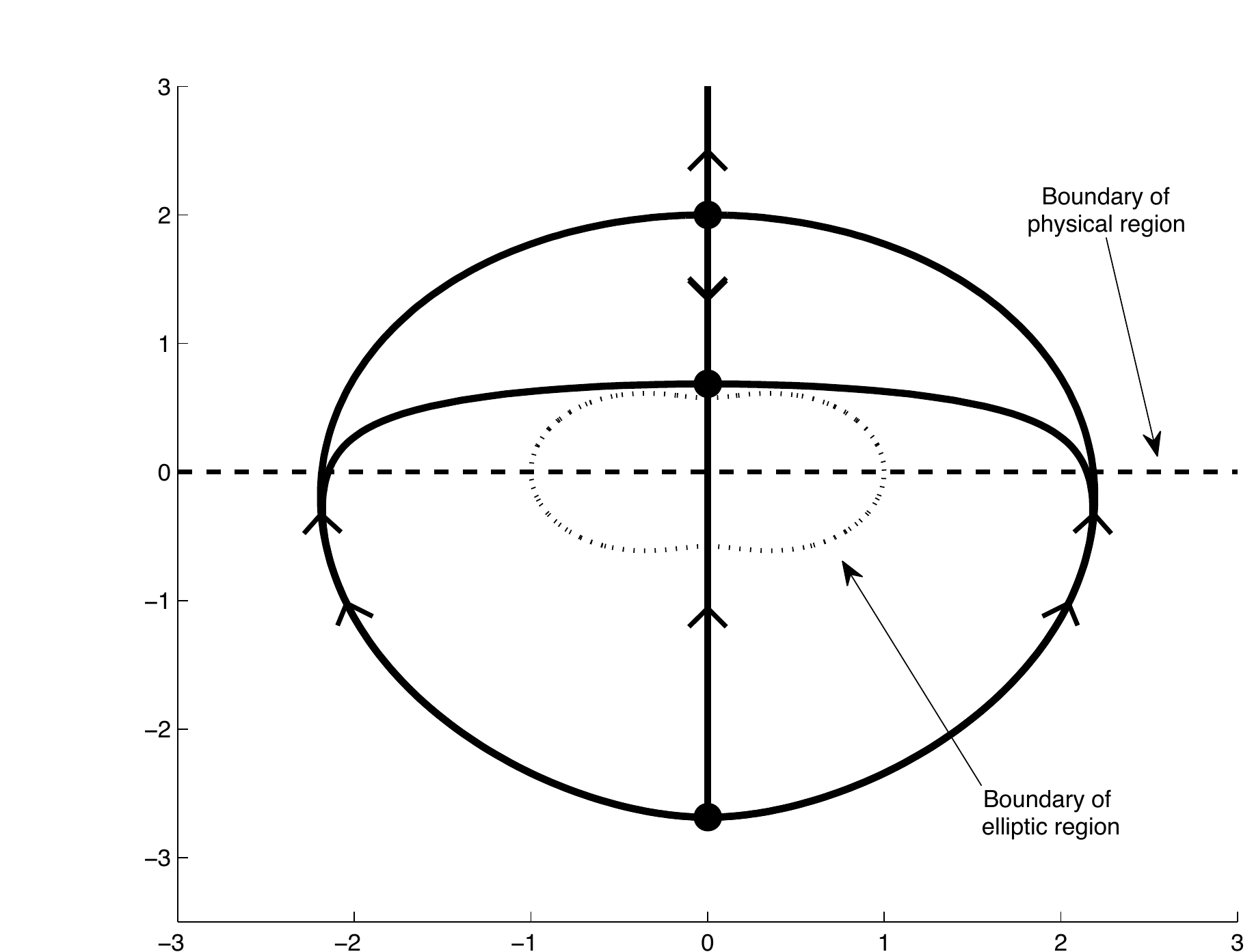}& (d) \includegraphics[scale=.4]{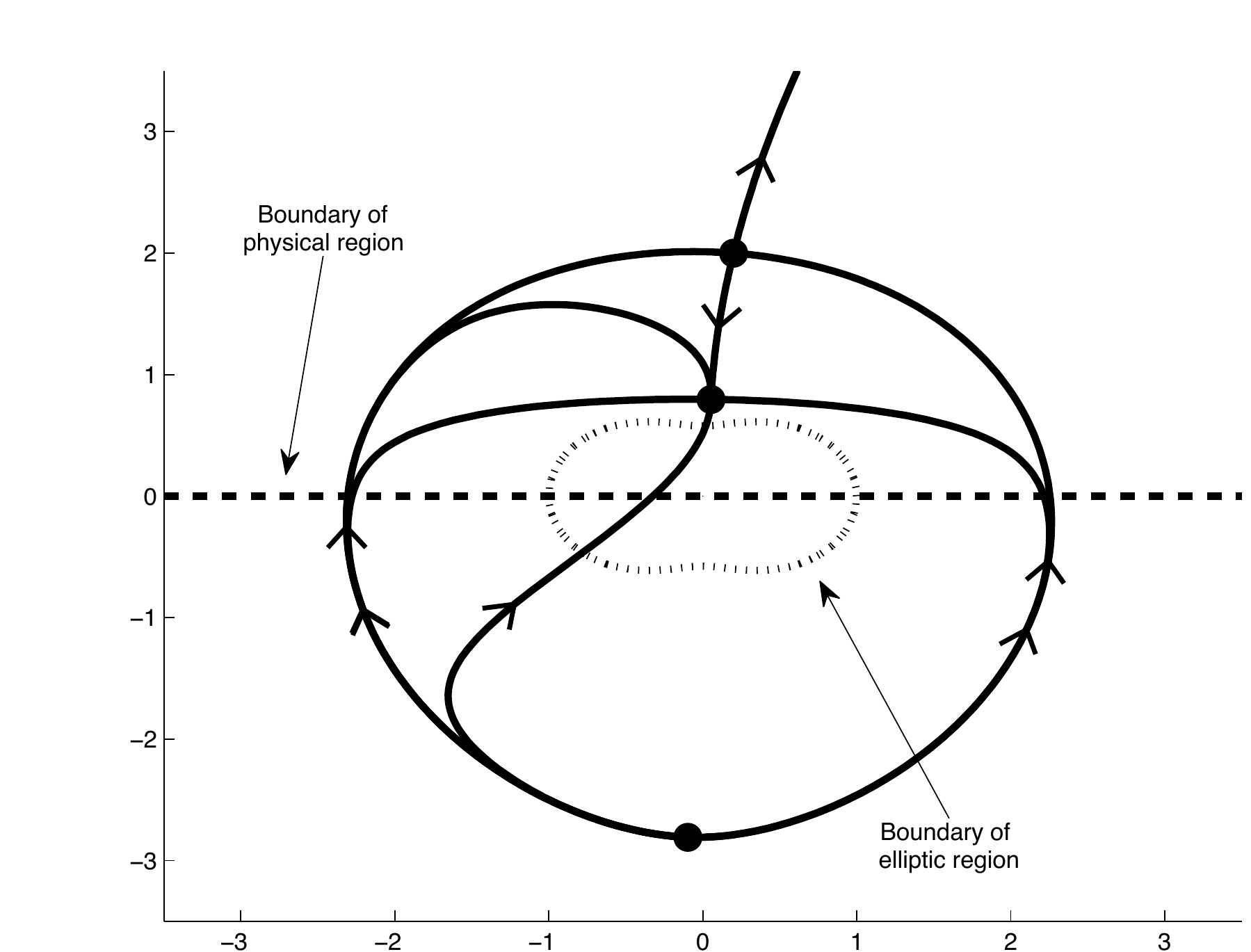}
\end{array}$
\caption{Typical three and five-equilibrium phase portraits for the
  2D compressible case. The dark dashed lines bound the physically
  relevant region $a_3>0$ and the light dashed lines surround  the
  region $m_2<0$  where the system (\ref{bfirstorder2D}) loses hyperbolicity, see Section \ref{333}.  In (a) we have $\alpha_2=0$, $\alpha_3=0.6$, and $\sigma=0.64$, in (b) $\alpha_2=0.2$, $\alpha_3=0.1$, and $\sigma=4$, in (c)  $\alpha_2=0$, $\alpha_3=2$, $\sigma=4.84$, and in (d) $\alpha_2=0.2$, $\alpha_3=2$, $\sigma=5.29$. }
\label{compressible_phase_portrait_1}
\end{figure}
\end{center}


\medskip

{\bf Case (ii) $\mathbf{\alpha_2\ne 0}$.}
Setting $x=|a_+|^2 - |\alpha|^2$ we see that \eqref{rhcomp2D} is
solved by:
$$ a_{2+}  = \frac{(|\alpha|^2-\sigma)}{ (x + |\alpha|^2-\sigma)}\alpha_2,
\qquad
a_{3+}   = \frac{(|\alpha|^2-1-\sigma)}{ (x + |\alpha|^2-1
  -\sigma)}\alpha_3. $$
Substituting into the definition of $x$ and rearranging we obtain:
\be\label{quint}
\begin{split}
(x+|\alpha|^2)(x+|\alpha|^2-\sigma)^2(x+|\alpha|^2-1-\sigma)^2
=&(|\alpha|^2-\sigma)^2(x+|\alpha|^2-1-\sigma)^2\alpha_2^2\\
&+(x+|\alpha|^2-\sigma)^2(|\alpha|^2-1-\sigma)^2\alpha_3^2,
\end{split}
\ee
yielding the quintic:
\be\label{yquint}
y(y-\sigma)^2(y-1-\sigma)^2=(|\alpha|^2-\sigma)^2(y-1-\sigma)^2\alpha_2^2
+(y-\sigma)^2(|\alpha|^2-1-\sigma)^2\alpha_3^2,
\ee
where $y=|a_+|^2$. From the roots of (\ref{yquint}) $a_+$ may be recovered through:
\ba\label{yrearranged}
a_{2+}   &= \frac{|\alpha|^2-\sigma}{y-\sigma}\alpha_2,
\qquad
a_{3+}   = \frac{|\alpha|^2-1-\sigma}{y-1 -\sigma}\alpha_3.\\
\ea
Evidently, these solutions are not 1d, as
$a_{2+}/a_{3+}\ne \alpha_2/\alpha_3$, unless $y=|\alpha|^2$,
in which case other roots of (\ref{yquint}) are:
$ \sigma$, $1+\sigma$, contradicting (\ref{yrearranged}). 

Recall that the nonphysical solutions with $a_{3+}\le 0$ are discarded and
that the further condition of
hyperbolicity of endstates is not necessary for existence of
profiles, but is needed to apply the basic stability framework
of Section \ref{nstab}.
We shall discuss profiles with nonhyperbolic endstates in Appendix \ref{s:complex}.
See Figures \ref{compressible_phase_portrait_1} (b) (d)
for a typical phase portrait computed numerically using MATLAB.


\subsection{The 3D compressible case - continued}\label{s:port3D}

In the full 3D compressible case, (\ref{sys}) can be written as:
\begin{equation*}
\begin{split}
\tilde a_{t} +  (|a|^2 \tilde a)_z &= \left(\frac{\tilde a_{z}}{a_3}\right)_z,\\
a_{3,t} +  ((|a|^2-1)a_3)_z &= 2\left(\frac{a_{3,z}}{a_3}\right)_z,
\end{split}
\end{equation*}
with $a = (\tilde a, a_3)$ and $\tilde a=(a_1,a_2)$.
The corresponding Rankine--Hugoniot relations read:
\ba\label{rhcomp3D}
(|a_+|^2-\sigma) \tilde a_+ -  (|\alpha|^2-\sigma) \tilde \alpha &= 0\\
(|a_+|^2-1-\sigma) a_{3+} -  (|\alpha|^2-1-\sigma) \alpha_3 &= 0
\ea
where $\alpha = (\tilde\alpha, \alpha_3)=a(-\infty)$ and $a_+ =
(\tilde a_+, a_{3+}) = a(+\infty)$ and $a_{3+}, \alpha_3>0$.
Also, by invariance with respect to rotations in the
$\tilde a$ plane, we will assume that $\alpha_1=0$, without loss of generality.

\medskip

{\bf Case (i) $\mathbf{\alpha_2=0}$.}
Here, the phase portrait can be easily deduced from that in 
(i) Section \ref{s:port2D} of the 2d compressible case.
That is, there is at most one physically feasible 1d
profile connecting to rest
point $a_+=(0,0,(-\alpha_3 + \sqrt{4(1+\sigma) - 3\alpha_3^2})/2)$, 
and a ring of rest points
$a_+=(r\cos \theta,r\sin \theta,a_{3+})$ with
$a_{3+}=  (1+\sigma-\alpha_3^2 )\alpha_3$ 
and $r=\pm\sqrt{\sigma -a_{3+}^2}$.
Again, we are ignoring possible equilibria in the nonphysical plane
$a_3=0$.

\medskip

{\bf Case (ii) $\mathbf{\alpha_2 \ne 0}$.}
This case includes the 2d portrait of case (ii) in Section
\ref{s:port2D} for the 2D compressible case when $a_1\equiv 0$.

If $a_{1+}\ne 0$, then $|a_+|^2=\sigma$,
and so $|\alpha_{2}|^2=\sigma$ (in view of the second equation in (\ref{rhcomp3D})).
Thus, except in this degenerate case, the set of equilibria is
{\it only} that of the planar case already treated.
The types of shock connections may be different than in Section
\ref{s:port2D}, and profiles
may go out of plane to yield new connections.

The above situation is quite reminiscent of the case of MHD \cite{BLZ}.
In particular, if $\alpha_2$ is varied slightly from the rotationally
symmetric situation $\alpha_2=0$, then one may conclude by persistence
of invariant sets as in \cite{FS} 
that ``Alfven''-type profiles must arise in the
rotationally degenerate characteristic field.

\section{Numerical stability analysis}

In this section, we describe the numerical 
Evans function method, based on the numerical approximation using the
polar-coordinate algorithm developed in \cite{HuZ}; 
see also \cite{BHRZ,HLZ,HLyZ,BHZ}.  
Since the Evans function is analytic in the region $\{Re \lambda\ge 0\}$
of interest, we can numerically compute its winding number
around a large semicircle $B(0,R)\cap \{Re \lambda \ge 0\}$,
enclosing all possible nonstable roots.  
This allows us to determine
stability through the Evans condition (D); alternatively,
as we shall do here, through its integrated version
$\rm (\tilde D)$ (resp., $\rm (\tilde D')$).
In the case of instability, one may go further to locate the roots
and study stability and bifurcation boundaries as model parameters are varied.
This approach was introduced in basic form by Evans and Feroe \cite{EF}
and it has since been elaborated and greatly generalized.
For applications to successively more complicated systems, see for example
\cite{PSW,AS,Br,BrZ,BDG,HuZ,HLZ,HLyZ,BHZ,BLZ}.

\subsection{The Evans systems}

Linearizing about a traveling wave solution $(\bar a, \bar b) = (\bar
a_1,\bar a_2, \bar a_3, \bar b_1, \bar b_2, \bar b_3)$
of \eqref{bfirstorder2}, we obtain the eigenvalue problem:
\begin{equation}\label{lineareigprob}
\begin{split}
\lambda a_j-sa_j'-b_j'&=0\qquad \mbox{ for } j=1,2\\
\lambda a_3-sa_3'-b_3'&=0\\
\lambda b_j -s b_j' -(|\bar a|^2a_j+2(\bar a\cdot a)\bar a_j)'&=( b_j'/\bar a_3-a_3\bar b_j'/\bar a_3^2)'\\
\lambda b_3-sb_3'-((|\bar a|^2-1)a_3+2(\bar a\cdot a)\bar a_3)'&
=2(b_3'/\bar a_3-a_3\bar b_3'/\bar a_3^2)'.
\end{split}
\end{equation}
We make the substitution $\tilde a_i(z) =\int_{-\infty}^z a_i(y)~\mbox{d}y$ 
and $\tilde b_i(z)=\int_{-\infty}^z b_i(y)~\mbox{d}y$
into \eqref{lineareigprob} and then integrate from $-\infty$ to $z$ 
to obtain, after dropping the tilde notation:
\begin{equation}\label{inteigprob}
\begin{split}
\lambda a_j-sa_j'-b_j'&=0\\
\lambda a_3-sa_3'-b_3'&=0\\
\lambda b_j-sb_j'-(|\bar a|^2 a_j'+2(\bar a\cdot a')\bar a_j)&=b_j''/\bar a_3 -a_3'\bar b_j'/\bar a_3^2\\
\lambda b_3-sb_3'-((|\bar a|^2-1)a_3'+2(\bar a\cdot a')\bar a_3)&=2(b_3''/\bar a_3-a_3'\bar b_3'/\bar a_3^2).
\end{split}
\end{equation}

\subsubsection{The 3D compressible case}

In the full 3D case (\ref{inteigprob}) may be written as a first order
system $Z'= \mathcal{A}(z,\lambda)Z$, 
with 
$$Z=(b_1,a_1,a_1',b_2,a_2,a_2',b_3,a_3,a_3')^T$$ 
and 
\begin{equation}
\mathcal{A}(z,\lambda)=
\left[\begin{array}{cc} B(z,\lambda) &C(z,\lambda)\\
D(z,\lambda)&E(z,\lambda) \end{array}\right],
\end{equation}
where: 
\begin{equation}\label{Bev}
\begin{split}
B(z,\lambda)=\left[\begin{array}{ccc}
0&\lambda&-s\\ \\
0&0&1\\ \\
\displaystyle{\frac{-\lambda \bar a_3}{s}}&\lambda \bar a_3
&\displaystyle{\frac{\lambda+\bar a_3 (|\bar a|^2-s^2+2\bar a_1^2)}{s}}
\end{array}\right], 
\end{split}
\end{equation}
\begin{equation*}
\begin{split}
C(z,\lambda)&=\left[\begin{array}{cccccc}
0&0&0&0&0&0\\ \\
0&0&0&0&0&0\\ \\
0&0&\displaystyle{\frac{2\bar a_1 \bar a_2 \bar a_3}{s}}&0&0
&\displaystyle{\frac{2\bar a_1 \bar a_3^3 -\bar b_1'}{s\bar a_3}} 
\end{array}\right], \\ \\
D(z,\lambda)&=\left[\begin{array}{cccccc}0&0&0&0&0&0\\ \\
0&0&0&0&0&0\\ \\
0&0&\displaystyle{\frac{2\bar a_1 \bar a_2 \bar a_3}{s}}&0&0
&\displaystyle{\frac{\bar a_1 \bar a_3^2}{s}}\end{array}\right]^T,
\end{split}
\end{equation*}
\begin{small}
\begin{equation}\label{2dcompevans}
{E}(z,\lambda)=
\left[
\begin{array}{cccccc}
0&\lambda&-s&0&0&0\\ \\
0&0&1&0&0&0\\ \\
\displaystyle{\frac{-\lambda \bar a_3}{s}}&\lambda \bar a_3
&\displaystyle{\frac{\lambda +\bar a_3(|\bar  a|^2-s^2+2\bar a_2^2)}{s}}
&0&0&\displaystyle{\frac{-\bar b_2'+2\bar a_2 \bar a_3^3}{s\bar a_3}}\\ \\
0&0&0&0&\lambda&-s\\ \\
0&0&0&0&0&1\\ \\
0&0&\displaystyle{\frac{\bar a_2\bar a_3^2}{s}}&\displaystyle{\frac{-\lambda \bar a_3}{2s}}
&\displaystyle{\frac{\lambda \bar a_3}{2}}
&\displaystyle{\frac{\bar a_3^2(|\bar a|^2-1-s^2+2\bar a_3^2)-2\bar b_3'+2\bar a_3 \lambda}{2s\bar a_3}}
\end{array}\right].
\end{equation}
\end{small}

\subsubsection{The 2D incompressible shear case}

For \eqref{shearfirstorder}, the same procedure as in
(\ref{lineareigprob}) yields:
\begin{equation}\label{sheareigenvalueeq}
\begin{split}
\lambda a - sa'-b'&=0,\\
\lambda b - sb' - ((1+|\bar a|^2)a + 2(\bar a\cdot a)\bar a)'&=b''.
\end{split}
\end{equation}
Substituting $\tilde a_i(z)=\int_{-\infty}^za_i(x)~\mbox{d}x$, $\tilde b_i(z)=\int_{-\infty}^z b_i(y)~\mbox{d}y$,
into \eqref{sheareigenvalueeq} and integrating from $-\infty$ to
$z$ we obtain, after dropping the tilde notation:
\begin{equation}\label{integratedeigenvalueproblem}
\begin{split}
\lambda a_i-sa_i'-b_i'&=0,  \qquad \mbox{ for } i=1,2\\
\lambda b_i-sb_i'-((1+|\bar a|^2)a_i'+ 2(\bar a \cdot a') \bar a_i) &=b_i''.
\end{split}
\end{equation}
Let $Z=(a_1,b_1,b_1',a_2,b_2,b_2')^T$.
Then \eqref{integratedeigenvalueproblem} may be written as (\ref{firstorderE}),
where:
\begin{small}
\begin{equation*}
\begin{aligned}
&\mathcal{A}(z,\lambda) =\\
&\;  \left[\begin{array}{cccccc}
\displaystyle{\frac{\lambda}{s}}&0&\displaystyle{\frac{-1}{s}}&0&0&0\\ \\
0&0&1&0&0&0\\ \\
\displaystyle{\frac{-\lambda(1+3\bar a_1^2+\bar a_2^2)}{s}}&\lambda
&\displaystyle{\frac{1-s^2+3\bar
  a_1^2+\bar a_2^2}{s}}&\displaystyle{\frac{-2\lambda\bar a_1\bar a_2}{s}}&
0&\displaystyle{\frac{2\bar a_1\bar a_2 }{s}}\\ \\
0&0&0&\displaystyle{\frac{\lambda}{s}}&0&-\displaystyle{\frac{1}{s}}\\ \\
         0&0&0&0&0&1\\ \\
-\displaystyle{\frac{2\lambda \bar a_1 \bar a_2 }{s}}&0&\displaystyle{\frac{2\bar a_1 \bar
a_2}{s}}&\displaystyle{\frac{-\lambda(1+\bar a_1^2+3\bar a_2^2)}{s}}&\lambda
&\displaystyle{\frac{1-s^2+\bar a_1^2+3\bar a_2^2}{s}}
\end{array}\right].
\end{aligned}
\end{equation*}
\end{small}

\begin{center}
\begin{figure}[htbp] 
$\begin{array}{lr}
 (a) \includegraphics[scale=.4]{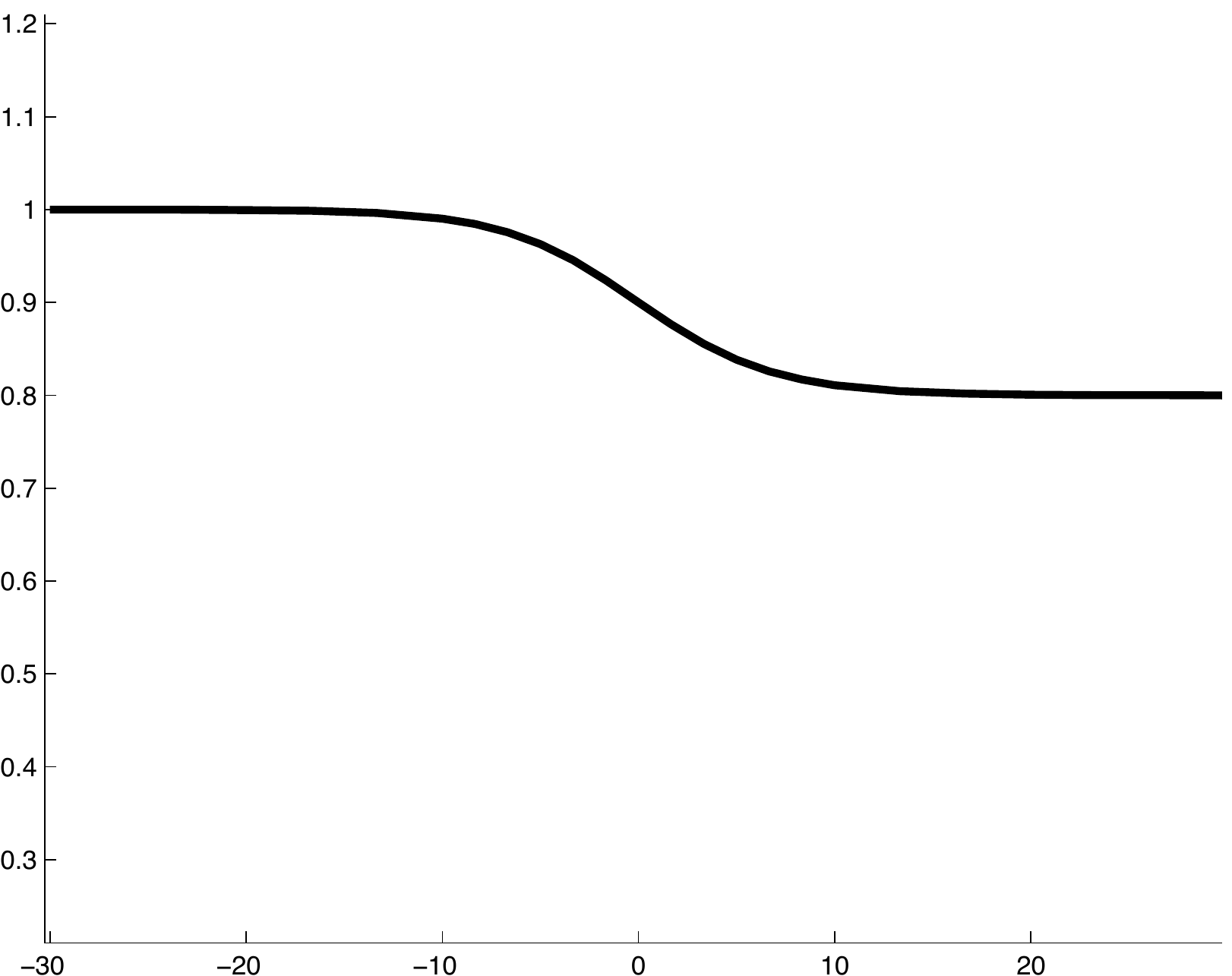}& (b) \includegraphics[scale=.4]{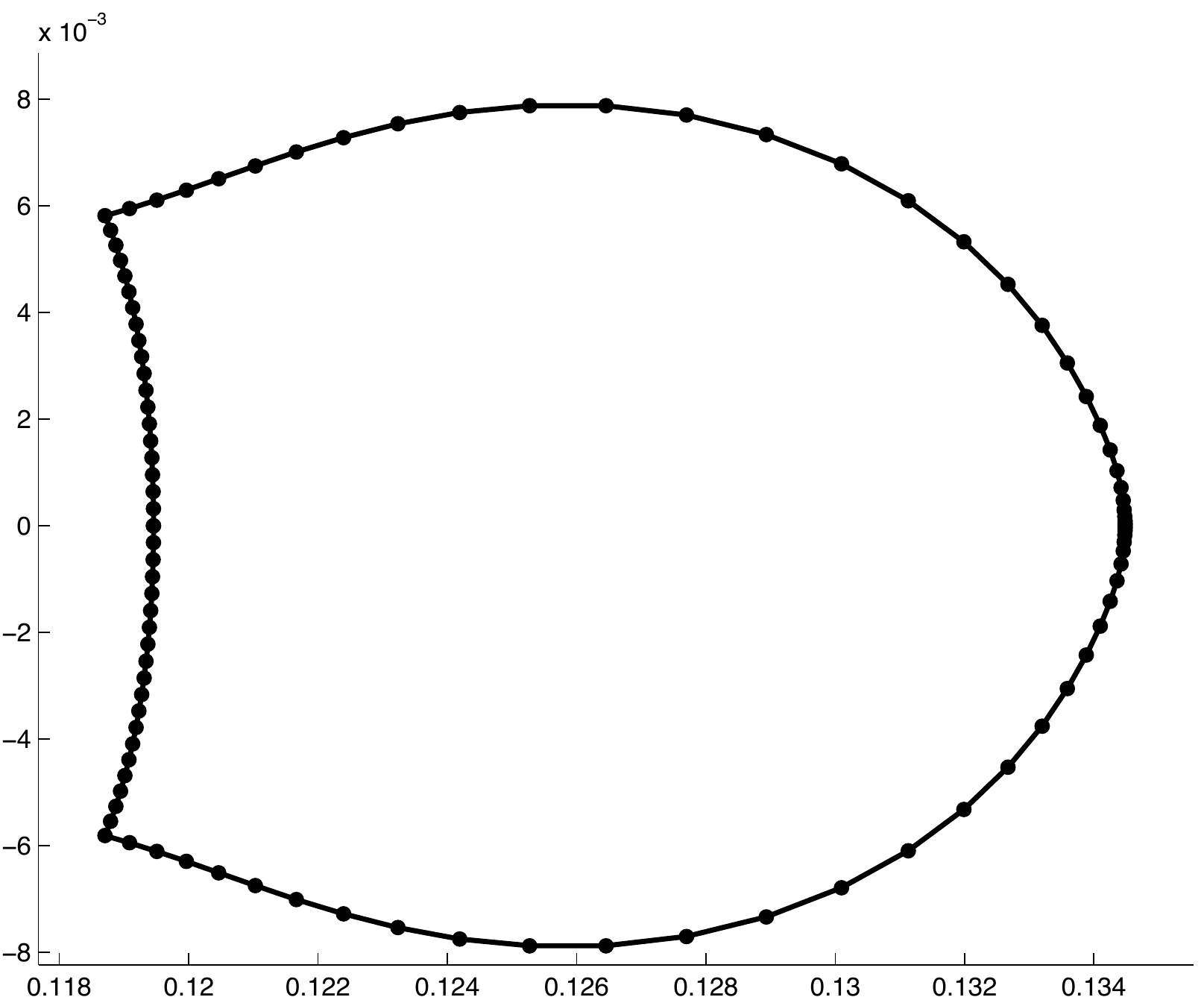}
 \end{array}$
 \caption{(a) Traveling wave profile $\bar V$ for the shear case with parameters values $\alpha_1=1$, $\alpha_2=0$, and $s=1.8547$ corresponding to a Lax shock connecting endstates $(1,0)$ an $(0.8,0)$. (b) The image of the semicircle under Evans function $\tilde D$.}\label{shear_evan_lax}
\end{figure}
\end{center}

\begin{center}
\begin{figure}[htbp] 
$\begin{array}{lr}
 (a) \includegraphics[scale=.4]{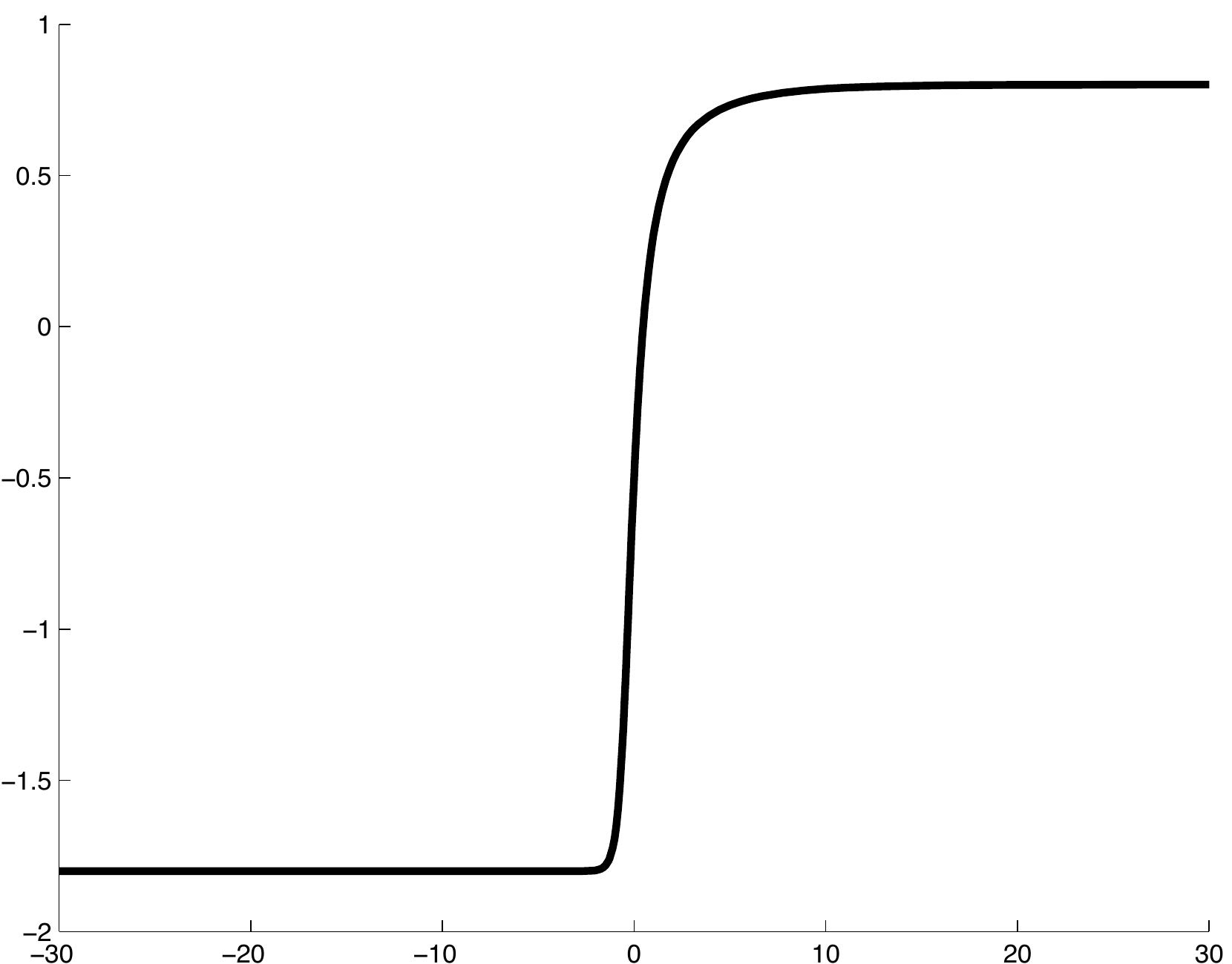}&(b) \includegraphics[scale=.4]{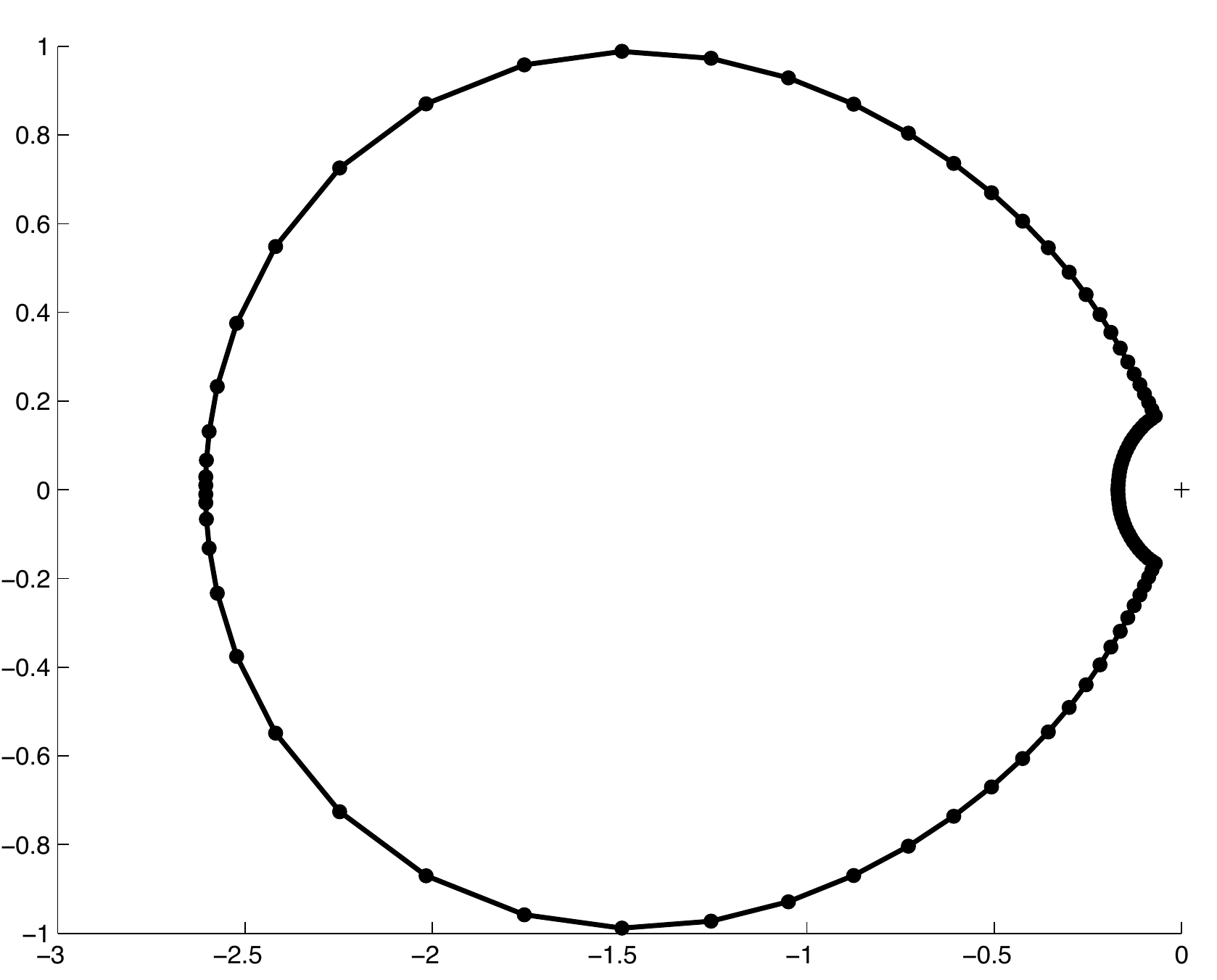}
 \end{array}$
 \caption{(a) Traveling wave profile $\bar V$  for the shear case with parameter values $\alpha_1=1$, $\alpha_2=0$, $s=1.8547$ corresponding to an overcompressive wave connecting endstates $(0.8,0)$ and $(-1.8,0)$. (b)   The image of the  semicircle under Evans function $\tilde D$.}\label{shear_evan_oc}
\end{figure}
\end{center}

\subsubsection{The 2D compressible case}

With $j=2$ in \eqref{inteigprob} and using $b_j'=\lambda a_j - s
a_j'$, \eqref{inteigprob} can be equivalently written
as (\ref{firstorderE}) 
with $Z=(b_2,a_2,a_2',b_3,a_3,a_3')^T$ and $\mathcal{A}(z,\lambda)
= E(z,\lambda)$ given in (\ref{2dcompevans}) where $a=(a_2, a_3)$.

\subsubsection{Transverse equations} 
Consider now a 2D compressible solution as a solution
of the full 3D system (\ref{bfirstorder2}). 
We find that the integrated eigenvalue equations (\ref{inteigprob})
decouple into the 2D equations plus the transverse system,
obtained from the equations corresponding to $j=1$ in system
(\ref{inteigprob}) after putting $\bar a_1 = \bar b_1 = 0$:
\ba\label{trans}
\lambda a_{1} -sa_1'-b_1'&=0,\\
\lambda b_{1} -sb_1'-|\bar a|^2 a_1'&=
\frac{b_1''}{\bar a_3},\\
\ea
The system (\ref{trans}) has the form of as (\ref{firstorderE}) with
$Z=(b_1, a_1, a_1')$ and $\mathcal{A}(z,\lambda)= B(z,\lambda)$ 
given in (\ref{Bev}), after putting $\bar a_1=0$. 



\subsection{Approximation of the profile and of the Evans function}
\label{profnum}

Following \cite{BHRZ,HLZ}, we approximate the traveling wave profile 
using one of \textsc{MATLAB}'s boundary-value solvers 
{\tt bvp4c} \cite{SGT}, {\tt bvp5c} \cite{KL}, or {\tt bvp6c} \cite{HM}.
These are adaptive Lobatto quadrature schemes that can be interchanged
for our purposes;  
for rigorous error/convergence bounds for such algorithms, 
see e.g. \cite{Be1,Be2}.
The calculations are performed on a finite computational domain 
$[-L,L]$,  
where the values of approximate plus and minus spatial infinity $L$ 
are determined experimentally by the requirement that the absolute error 
$|\bar V(\pm L)-V_\pm|\leq TOL$ be within a prescribed tolerance, 
say $TOL=10^{-3}$. Here $\bar V$ and $V_\pm$ are the profile and limiting endstates
as defined in Section \ref{s:stabtheory}.


\smallskip

Using now the notation of section \ref{s:stabtheory}, define for $z\geq0$ and $z\leq 0$, respectively:
$$\mathcal{Z}^+(z,\lambda) = Z^+_1(z,\lambda) \wedge \ldots \wedge
Z^+_k(z,\lambda) \quad \mbox{and}  \quad
\mathcal{Z}^-(z,\lambda) = Z^-_{k+1}(z,\lambda) \wedge \ldots \wedge
Z^-_N(z,\lambda), $$
so that the Evans function is given by:
\begin{equation}\label{Dhere}
D(\lambda) = \mathcal{Z}^+(0,\lambda) \wedge \mathcal{Z}^-(0,\lambda).
\end{equation}
Since for $L>0$ large, $P_\pm(\pm L,\lambda)$ approximately equals
$\mbox{Id}$, we obtain:
\begin{equation}\label{newno}
Z_i^\pm(\pm L,\lambda) \sim e^{\pm\mathcal{A}_\pm(\lambda) L}
\tilde Z_i^\pm(\lambda),
\end{equation}
where the ``$+$''  sign is taken with indices
$i=1\ldots k$ and the ``$-$'' sign with $i=k+1\ldots N$. 
Recall that $\mathcal{S}(\lambda)=span\{\tilde Z_i^+(\lambda)\}_{i=1..k}$ is
the stable space of $\mathcal{A}_+(\lambda)$, while
$\mathcal{U}(\lambda)=span \{\tilde Z_i^-(\lambda)\}_{i=k+1..N}$ represents 
the unstable space of $\mathcal{A}_-(\lambda)$.

The analytic bases $\{\tilde Z_i^\pm(\lambda)\}_i$ are obtained by the
following procedure \cite{HuZ, Z2}. First, one computes the
eigenprojections $\mathcal{P}_+(\lambda)$ and $\mathcal{P}_-(\lambda)$
onto, respectively, the space $\mathcal{S}(\lambda)$ and $\mathcal{U}(\lambda)$.
This can be done by setting:
$$\mathcal{P}_\pm = R_\pm \big(L_\pm R_\pm\big)^{-1}L_\pm,$$
where $R_\pm(\lambda)$ and $L_\pm(\lambda)$ are matrices consisting of any orthonormal
right and left bases of $\mathcal{S}(\lambda)$ (when with subscript
``+'') and $\mathcal{U}(\lambda)$ (when with subscript ``--''). Then, a
standard result in matrix perturbation theory \cite{K} states that the
analytic in $\lambda$ bases $\tilde{\mathcal{Z}}^\pm(\lambda)$ can be
prescribed constructively as the solution of Kato's ODE:
$$(\tilde{\mathcal{Z}}^\pm)' = \Big( \mathcal{P}_\pm' \mathcal{P}_\pm
- \mathcal{P}_\pm \mathcal{P}_\pm' \Big) \tilde{\mathcal{Z}}^\pm, 
\qquad \tilde{\mathcal{Z}}^\pm(\lambda_0) = R_\pm(\lambda_0),$$
where $'$ denotes the differentiation with respect to $\lambda$.
This prescription is also minimal in the sense that $\mathcal{P}_\pm R_\pm'=0$.

We now continue the construction of the approximate Evans
function.  As a consequence of (\ref{newno}), for large $L$ we set:
\begin{equation*}
\begin{split}
\mathcal{Z}^+(L,\lambda) \sim ~&\mathcal{Z}^+_{app}(L,\lambda) :=
e^{\mbox{tr} (\mathcal{A}_+(\lambda)_{|\mathcal{S}(\lambda)})L} \tilde
Z_1^+(\lambda) \wedge\ldots\wedge \tilde Z_k^+(\lambda),\\
\mathcal{Z}^-(-L,\lambda) \sim ~ &\mathcal{Z}^-_{app}(-L,\lambda) :=
e^{-\mbox{tr} (\mathcal{A}_-(\lambda)_{|\mathcal{U}(\lambda)})L} \tilde
Z_{k+1}^-(\lambda) \wedge\ldots\wedge \tilde Z_N^-(\lambda).
\end{split}
\end{equation*}
The objective is now to trace the evolution of the differential form
$\mathcal{Z}^+_{app}(\cdot, \lambda)$ backward in $z$, and the evolution
of $\mathcal{Z}^-_{app}(\cdot, \lambda)$ forward in $z$, starting
from, respectively,
the initial data $\mathcal{Z}^+_{app}(L, \lambda)$ and
$\mathcal{Z}^-_{app}(-L, \lambda)$, and according to
the system as in  (\ref{firstorderE}):
\begin{equation}\label{Zapp}
\mathcal{Z}_{app}'(z,\lambda) = \mathcal{A}(z,\lambda)
\mathcal{Z}_{app}(z,\lambda). 
\end{equation}
The numerical approximation of
$D(\lambda)$ in (\ref{Dhere}) is then recovered through:
\begin{equation}\label{Dapp}
D(\lambda) \sim D_{app}(\lambda)
:= \mathcal{Z}^+_{app}(0,\lambda) \wedge
\mathcal{Z}^-_{app}(0,\lambda).
\end{equation}

To solve (\ref{Zapp}) for $\mathcal{Z}^\pm_{app}$ we use
the polar-coordinate method described in \cite{HuZ}, 
which encodes $\mathcal{Z}^\pm_{app}$ as product of a complex scalar $r^\pm$ and the exterior
product $\Omega^\pm$ of an orthonormal basis 
$\{\omega_i^+\}$ of $\mathcal{S}$ or, respectively, an orthonormal
basis $\{\omega_i^-\}$ of $\mathcal{U}$:
$$\mathcal{Z}_{app}^\pm(z,\lambda)=r^\pm(z,\lambda)
\Omega^\pm(z,\lambda),\qquad \Omega^+ = \omega_1^+\wedge\ldots\wedge
\omega_k^+, \quad \Omega^- = \omega_{k+1}^-\wedge\ldots\wedge
\omega_N^-.$$
The above quantities $\Omega$ evolve 
by some implementation (e.g. Drury's method below) of continuous 
orthogonalization, where the ``radius'' $r$ satisfies 
a scalar ODE slaved to $\Omega$, related to Abel's formula 
for evolution of a full Wronskian. Namely (\ref{Zapp}),
is equivalent to:
\begin{equation}\label{Drury}
\begin{split}
\Omega'(z,\lambda) & = \Big(\mbox{Id}_N - \Omega\Omega^* \Big) 
\mathcal{A}(z,\lambda) \Omega(z,\lambda)\\
r'(z,\lambda) &= \mbox{tr}\Big(\Omega^*\mathcal{A}(z,\lambda)\Omega\Big)\cdot r(z,\lambda)
\end{split}
\end{equation}
and we recover, in view of (\ref{Dapp}): 
$$D_{app}(\lambda) =  r^+(0,\lambda) r^-(0,\lambda)
\cdot \Omega^+(0,\lambda)\wedge\Omega^-(0,\lambda),$$
see \cite{HuZ,Z2,Z3} for further details. The rationale for solving
the system (\ref{Drury}) for the decomposition of $\mathcal{Z}_{app}$, rather
than the original (\ref{Zapp}) is that the imposition of orthonormality
on $\Omega$ prevents the collapse of the various columns (solutions)
onto a single fastest-growing mode, as would otherwise be the case.
For a discussion of this and other numerical issues connected with
the polar coordinate method, see \cite{HuZ,Z3}.


The calculations of (\ref{Drury}) for individual $\lambda$ are carried out 
using \textsc{MATLAB}'s {\tt ode45} routine, an
adaptive 4th-order Runge-Kutta-Fehlberg method (RKF45)
with excellent accuracy and automatic error control.  
Typical runs involved roughly $60$ mesh points per side, 
with error tolerance set to {\tt AbsTol = 1e-8} and {\tt RelTol = 1e-6}.  
To produce analytically varying Evans function output, the initializing
bases $\{\tilde Z^\pm_i\}$ are chosen analytically using Kato's ODE
\cite{GZ,HuZ,BrZ,BHZ}.
Numerical integration of Kato's ODE is carried out using a
second-order algorithm introduced in \cite{Z2,Z3}. 

\subsection{Winding number computation}
\label{windingalg}
Recall that the Evans condition amounts to
checking for the existence of unstable zeros of the 
integrated Evans function $\tilde D$, described in section \ref{s:Evansint}.
We first observe (Proposition 6.10 \cite{HLZ}),
that for shock profiles of the hyperbolic--parabolic systems of the type
we consider, there holds:
\begin{equation}\label{Rlimit}
\lim_{|\lambda|\to \infty} \frac{\tilde D(\lambda)}{e^{\alpha \sqrt{\lambda}}}=C
\qquad \hbox{\rm uniformly on } \; Re ~\lambda\ge 0,
\end{equation}
with constants $\alpha$ and $C\neq 0$.
When $\tilde D$ is initialized in the standard way on the real axis, 
so that $\tilde D(\lambda)=\tilde D(\bar \lambda)$, $\alpha$
and $C$ are necessarily real.  
The knowledge that limit in \eqref{Rlimit} exists allows actually to
determine $\alpha$ and $C$ by curve fitting of $\log \tilde
D(\lambda)=\log C+\alpha \lambda^{1/2}$  with respect to
$\lambda^{1/2}$, for large $|\lambda|$. 

One further determines the radius $R>0$ so that:
$$\tilde D(\lambda)\neq 0 \qquad \mbox{ for } |\lambda|\geq R 
\mbox{ and } Re~\lambda \geq 0,$$
by taking $R$ to be a value for which the relative error
between $\tilde D(\lambda)$ and $Ce^{\alpha \sqrt{\lambda}}$ becomes
less than $0.2$ on the entire semicircle:
$$ S_R = \partial \Big( B(0,R)\cap \{Re ~\lambda \ge 0\} \Big), $$
indicating sufficient convergence to ensure nonvanishing.  
For many parameter combinations, $R=2$ was sufficiently large. 
Alternatively, we could use energy estimates or direct tracking bounds
as in \cite{HLZ} and \cite{HLyZ}, respectively, to eliminate
the possibility of eigenvalues of sufficiently high frequency.
However, we have found the convergence study to be much more efficient
in practice; see \cite{HLyZ}.

We now compute the winding number $I(R)$ of the image curve $\tilde
D(S_R)$ with respect to $0$, which equals the degree of the
$2$d vector field given by $\tilde D$ in the interior region of $S_R$.
Since the index of any nondegenerate zero of a
holomorphic function is $+1$,  the condition 
$$I(R)=0$$ 
is hence equivalent to $\tilde D$ having no zeros in the open interior of the curve $S_R$.
Since all the shocks considered here are of Lax or overcompressive type,
condition $I(R)=0$ is equivalent to the Evans stability condition ($\tilde D$).

The winding number $I(R)$ is now computed by varying values of $\lambda$ 
along $20$ points of the contour $S$, with mesh size taken quadratic
in modulus to concentrate sample points near the origin where angles 
change more quickly, and summing the resulting changes in 
${\rm arg}(\tilde D(\lambda))$, using $\Im \log \tilde D(\lambda) 
= {\rm arg} \tilde D(\lambda) ({\rm mod} 2\pi)$.
To ensure winding number accuracy, we test a posteriori that the change in 
$\tilde D$ for each step is less than $0.2$, and add mesh points as
necessary to achieve this.  
(Recall, by Rouch\'e's Theorem, that accuracy is preserved 
so long as relative variation of $\tilde D$ along each mesh interval is 
$\le 1.0$.) 
In Tables \ref{tb_error} and \ref{tb_error2}  we give the radius of the domain
contour,  the number of mesh points, the relative error for change 
in argument of $\tilde D(\lambda)$ between steps, and the 
numerical approximation of spatial infinity $\pm L$.

\begin{center}
\begin{table}[!b]
\begin{tabular}{|c|c|c|c|c|c||c|c|c|c|c|c|}
\hline
$\alpha$&$s$&$R$&points&error&$L$      &   $\alpha$&$s$&$R$&points&error&$L$\\    \hline
0.2&1.8&2&38&0.1947&16.25&1&1.8&2&27&0.1787&6.3
  \\
  \hline
  1&2.8&2&21&0.1791&2.5&2&2.8&2&21&0.1878&3\\
  \hline
  3&3.8&2&20&0.1103&1.8&5&5.8&2&20&0.0903&1.05\\
\hline
\end{tabular}
\caption{Table demonstrating contour radius, number of mesh points, relative error, and spatial domain for the incompressible case.}
\label{tb_error2}
\end{table}
\end{center}


\begin{center}
\begin{table}[!b]
\begin{small}
\begin{tabular}{|c|c|c||c|c|c|c||c|c|c||c|c|c|c|}
\hline
$\alpha_2$&$\alpha_3$&$s$&$R$&points&error&$L$     &$\alpha_2$&$\alpha_3$&$s$&$R$&points&rel error&$L$\\
\hline
 0.1&1&1.9&2&20&0.13&15.01         & 0.1&6.6&9.5&2&20&0.00&2.01\\
\hline 
   0.9&2.6&8.7&2&20&0.01&2.01    & 1.3&3.4&8.3&2&20&0.01&2.01\\
\hline 
   3.7&4.2&7.1&2&20&0.02&2.01              & 1.7&0.2&3.9&2&20&0.08&19.01\\
\hline 
6.1&2.6&8.7&2&20&0.01&2.01      &4.1&4.6&8.3&2&20&0.01&2.01\\
\hline 
6.9&0.6&8.3&4&20&0.17&7.01       &6.9&0.2&8.3&2&20&0.12&15.01\\
\hline
\end{tabular}
\end{small}
\caption{Table demonstrating contour radius, number of mesh points, relative error, and spatial domain. The data on the left side corresponds to the compressible 2D system, and that on the right to the transverse system.}
\label{tb_error}
\end{table}
\end{center}

\subsection{Results of numerical experiments}\label{s:numresults}

In our numerical study, we sampled from a broad range
of parameters and checked stability of the resulting Lax and
over-compressive profiles whenever their endstates fell into
the hyperbolic region as required by our stability framework.
We did not find any undercompressive profiles for the model considered
here, either in the incompressible shear or the compressible case,
nor did Antman and Malek--Madani find undercompressive profiles in
their investigations of the incompressible shear case \cite{AM}.
As shown in Section \ref{nonexist}, undercompressive connections
cannot occur in the incompressible shear case for any choice of
potential.
However, we do not see why they could not occur for other choices of elastic
potential in the compressible case.
%

All our computations yielded zero winding number, consistent with
stability. All together, our study consisted of over 8,000 Evans function computations.
The following parameter combinations were examined for Evans stability.

\smallskip

\noindent \textbf{The 2D incompressible shear case.}
The following parameter combinations yielded Evans function output
with winding number zero, consistent with stability:
 \begin{align*}
&(\alpha,s) \in \{0.2:0.2:5\} \times \{0.2:0.2:7\}.
\end{align*}

\smallskip

\noindent
\textbf{The 2D compressible case.}
In the compressible 2D case and the transverse case following, we
computed the Evans function for the stated parameter combinations
whenever the profile end-states did not lie in the elliptic region. 
For $\alpha_2\neq 0$, we  restricted our attention to the
profiles connecting rest-points corresponding to 
solutions of \eqref{yquint} in the interval $[-50,50]$. 

We computed the Evans function for all 2 point configurations 
(Lax connections) coming from the following parameter combinations. 
All computations yielded zero winding number, consistent with stability:
\begin{align*}
&(\alpha_2,\alpha_3,s) \in \{0\}
 \times \{1.1:0.5:25.6\}
 \times \{0.5:0.5:20\}.
\\
&(\alpha_2,\alpha_3,s) \in \{0.1:0.4:25.3\}
\times \{0.2:0.4:25.4\}
 \times \{0.3:0.4:25.5\}.
\end{align*}
For the following parameter combinations, 
we investigated the 4 point configurations computing 
all Lax connections, and 5 overcompressive connections passing 
through evenly spaced points on the segment in phase space
connecting the saddle points:
\begin{align*}
&(\alpha_2,\alpha_3,s) \in \{0\} \times \{0.03:0.07:1.03\} \times \{0.05:0.07:1.05\},\\
&(\alpha_2,\alpha_3,s)=\{(0.08,0.59,0.75),(0.08,0.87,0.82),(0.22,0.66,0.75)\}.
\end{align*}

\smallskip

\noindent
\textbf{The transverse case.}
We computed the  transverse Evans function for the two point
configurations (Lax connections) for the following parameter
combinations:
\begin{align*}
&(\alpha_2,\alpha_3,s) \in \{0.1:0.4:25.3\} \times \{0.2:0.4:25.4\} \times \{0.3:0.4:25.5\}.
\end{align*}
In addition we examined the 4 point configuration corresponding to
$(\alpha_2,\alpha_3,s)=(0.1,0.8,0.8)$ computing the Evans function for
the 4 Lax connections and for 5 overcompressive connections passing
through points evenly spaced along the line in phase space between the
two saddle points. 

\subsubsection{Numerical performance}\label{s:perf}

The Evans function computations for the most part worked reliably
and well, showing performance comparable to that seen in previous
studies for gas dynamics \cite{HLZ,HLyZ} and MHD \cite{BHZ,BLZ}.
A typical winding number computation for a single profile 
took approximately 30 seconds and 
computation of the profile approximately 5 seconds.
%
%
%

As expected, performance degraded catastrophically
in various boundary situations: the small-amplitude limit as
$\frac{|a_+-a_-|}{|a_+|+|a_-|}\to 0$; the characteristic limit
as one or more characteristic speeds approach the shock speed;
the large-amplitude limit as $|a_\pm|$ approach infinity or
$a_3$ approaches the physical (infinite compression) boundary $a_3=0$; 
and the elliptic limit as one or both endstates $a_\pm$ 
approach the elliptic region where characteristic speeds are complex.
For discussion of causes of and (partial) cures for these numerical issues,
see, e.g., \cite{HLZ,BHZ,BLZ,Z3}. 
In the present study, such boundary cases were 
omitted.

%
%
%
%
%

\section{Discussion and open problems}
In this paper, we have obtained the first analytical stability results
for viscoelastic shock waves, stability of small-amplitude Lax shocks,
and set up a theoretical framework for future numerical and analytical studies
of shock waves of essentially arbitrary viscoelastic models.
A large-scale numerical Evans study for the canonical model
\eqref{part}  yielded a result of numerical stability for each 
of the more than 8,000  profiles tested, of both classical
Lax and nonclassical overcompressive type, and with amplitudes varying
from near zero to 50.

Interesting problems for the future are the treatment of more realistic potentials with physically correct asymptotic behavior, 
systematic numerical and asymptotic investigation across parameters
as in \cite{HLZ,HLyZ,BHZ,BLZ},
and the treatment of phase transitional elasticity
by incorporation of dispersive surface energy terms.

%
%

\appendix

\section{Appendix: General facts}\label{s:facts}
Though the investigations of this paper were carried out for
special choices of $W$, $\mathcal{Z}$, the methods we use apply
to much more general choices.  With an eye toward future work,
we collect in this appendix the information needed to carry
out such extensions.

\subsection{General elastic potential}\label{s:general}

\begin{theorem}\label{Wform}
Let $W:\mathbb{R}^{3\times 3}\longrightarrow \overline{\mathbb{R}}_+$
satisfy (\ref{frame_inv}) and (\ref{isotropic}). Then there exists
a scalar function $\sigma: \mathbb{R}^3\longrightarrow
\overline{\mathbb{R}}_+$, such that:
$$W(F) = \sigma(|F|^2, |FF^T|^2, \det F) .$$
The derivative $DW(F)\in \mathbb{R}^{3\times 3}$, wherever defined at
$F\in \mathbb{R}^{3\times 3}$
(so that $\partial_A W(F) = DW(F):A$), is given by:
$$ DW(F) =  \nabla \sigma(|F|^2, |FF^T|^2, \det F)\cdot \Big(2F, 4FF^TF,
\mathrm{cof }F\Big).$$
If $W(\mathrm{Id})=0$ and $W$ is $\mathcal{C}^2$ in a neighborhood of
$SO(3)$, then:
$$DW(\mathrm{Id}) = 0, \qquad
D^2W(\mathrm{Id}): A = \lambda (\mathrm{ tr }A)\mathrm{Id} + \mu \mathrm{ sym }A
\qquad \forall A\in\mathbb{R}^{3\times 3},$$
with the convention 
$\partial^2_{A_1, A}W(\mathrm{Id}) = (D^2W(\mathrm{Id}):A):A_1$ and 
the Lam\'e constants $\lambda$ and $\mu$:
$$\lambda = \nabla^2\sigma(3,3,1): \Big((2,4,1)\otimes (2,4,1)\Big),\quad
\mu = \nabla\sigma (3,3,1)\cdot (0,8,-2)$$
satisfying: $\mu\geq 0$ and $3\lambda+\mu\geq 0$.
\end{theorem}

\begin{proof}
According to the representation theorem \cite{TN}, every frame invariant and
isotropic $W$ depends only on the principal invariants 
of the left Cauchy deformation tensor $FF^T$,
that is $W(F) = \bar{\sigma}(\mbox{tr}(FF^T), \mbox{tr cof }(FF^T),
\det(FF^T))$. Since $\mbox{tr cof } Q = \frac{1}{2}(\mbox{tr }Q)^2 
- \frac{1}{2} \mbox{tr } (Q^2)$, 
the claim on the form of $W$ follows directly.

The formula for derivative $DW(F)$ follows from:
$$\partial_A |F|^2 = 2F:A,\quad \partial_A |FF^T|^2= 4FF^TF:A,
\quad \partial_A \mbox{det}F = \mbox{cof } F:A,$$
where for the last expression we used 
$\det(F+Q) = \det F + F:\mbox{cof } Q + Q:\mbox{cof } F + \det Q$,
 valid for $3\times 3$ matrices $F,Q$.

The vanishing of $DW(\mbox{Id})$ is clear since $W$ is minimized at
$\mbox{Id}$. The formula for $D^2W(\mbox{Id})$ follows by 
chain rule and Lemma \ref{lem1}.
Further, notice that $\partial^2 W_{A,A}(\mbox{Id}) \geq 0$, which reads:
\begin{equation}\label{p1}
\lambda |\mbox{tr }A|^2 + \mu |\mbox{ sym }A|^2\geq 0
\qquad \forall A\in \mathbb{R}^{3\times 3}.
\end{equation}
Evaluating (\ref{p1}) first at a traceless $A$ and then at $A=\mbox{Id}$ we see
that:
\begin{equation}\label{p2}
\mu\geq 0, \qquad 3\lambda+\mu\geq 0.
\end{equation}
To prove that (\ref{p2}) implies (\ref{p1}), write $ \mbox{sym }A $ a the sum
of orthogonal
matrices: $\mbox{sym } A = \mbox{diag} (a_{11}, a_{22}, a_{33}) + B$. Then:
$|\mbox{sym } A|^2 = \sum_{i=1}^3 a_{ii}^2 + |B|^2$, so:
$$\lambda |\mbox{tr }A|^2 + \mu |\mbox{ sym }A|^2 =
\lambda (\sum_{i=1}^3 a_{ii})^2 + \mu  \sum_{i=1}^3 a_{ii}^2 +\mu |B|^2
\geq (\lambda+\mu/3)(\sum_{i=1}^3 a_{ii})^2  + \mu |B|^2,$$
which ends the proof.
\end{proof}

For the behavior of $W$ close to the energy well $SO(3)$
it is important to know the derivatives of $W$ at
$R\in SO(3)$. It follows by frame invariance that:
$$\partial^n_{RF_1,\ldots RF_n}W (R) = \partial^n_{F_1,\ldots F_n}W
(\mathrm{Id})
\qquad \forall F_1\ldots F_n\in\mathbb{R}^{3\times 3}\quad\forall R\in SO(3).$$
Hence, it suffices to find the derivatives of $W$ at $\mathrm{Id}$.
Direct calculation yields the following:

\begin{lemma}\label{lem1}
For $F\in\mathbb{R}^{3\times 3}$, let $\alpha(F) = |F|^2$, $\beta(F)=|FF^T|^2$
and $\gamma(F) = \det F$. Then, for any $A_1, A\in \mathbb{R}^{3\times 3}$ we
have:
\begin{equation*}
\begin{split}
(\alpha, \beta,\gamma)(\mathrm{Id}) &= (3,3,1)\\
\partial_{A_1}(\alpha, \beta,\gamma)(\mathrm{Id})&= \Big(\mathrm{Id}:A_1\Big)
(2,4,1) \\
\partial^2_{A_1,A}(\alpha, \beta,\gamma)(\mathrm{Id})&= \Big(A:A_1\Big) (2,4,1)
+  \Big(\mathrm{sym} A:A_1\Big) (0,8,-2)\\
\partial^3_{A_1,A,A}(\alpha, \beta,\gamma)(\mathrm{Id})&= - \Big(\mathrm{cof
}A:A_1\Big) (0,8,-2)
+  \Big((\mathrm{cof } A + AA^T + 2A^2):A_1\Big) (0,8,0)\\
\partial^4_{A_1,A,A,A}(\alpha, \beta,\gamma)(\mathrm{Id})&=
3\Big((AA^TA):A_1\Big) (0,8,0)
\end{split}
\end{equation*}
\end{lemma}

For $F$ as in (\ref{Fplanar}), we have:
$$\alpha(F) = 2+|a|^2,\quad \beta(F) = 2+ |a|^4 + 2(|a|^2 - a_3^2), 
\quad \gamma(F) = a_3.$$
Hence, and without loss of generality, the reduced function $W(a)$
must be of the form $W(a) = \tilde\sigma(|a|^2, a_3)$, for some
$\tilde\sigma:\mathbb{R}^2\longrightarrow \overline{\mathbb{R}}_+$.

In the incompressible shear case, it reduces to:
\begin{equation}\label{ex23}
\check W(a)=\check \sigma(|a|^2)=\tilde \sigma(|a|^2,1),
\end{equation}
leading to a profile equation agreeing with that of 
the $2\times 2$ rotationally symmetric model:
\begin{equation}\label{A3}
a_{t} +( 2\nabla\check \sigma (|a|^2) a)_z= a_{zz}.
\end{equation}
This clarifies and puts in a more familiar context the
investigations of Antman and Malek-Madani \cite{AM} 
on existence of viscous profiles for the 2D incompressible shear model.

\subsection{General viscous stress tensor}\label{s:genvisc}
We do not have a complete categorization of possible $\mathcal{Z}$
satisfying (i)--(iii) corresponding to that of Theorem \ref{Wform} for
the elastic energy density $W$.
However, we note the related discussion of Antman \cite{A} for
the class of systems of {\it strain-rate type}:
$$ \mathcal Z(F,Q)= FS(C,D),$$
where $S$ is a symmetric dissipation tensor depending on the the metric
$C=F^TF$ and on its time derivative $D = F^TQ+ Q^T F=2\mbox{sym}(F^TQ)$.
These automatically satisfy (i) and (ii) because
$\mbox{sym}\big( (RF)^T(RK F + RQ)\big) =
\mbox{sym}( F^TQ + F^TKF) =\mbox{sym}(F^TQ)$. Thus only (iii) need be
checked, in the form:
\be\label{betteriii}
S(C,D): D\geq 0.
\ee
Both of the examples \eqref{Zs} are of this type.
For $\mathcal{Z}_1=FS_1$, we have:
$$ S_1(C,D)= \mbox{sym}(F^TQ)=\frac{1}{2}D,$$
which evidently satisfies the strict version of inequality \eqref{betteriii}:
\be\label{strict}
S(C,D): D \ge \gamma |D|^2, \quad \gamma>0.
\ee
For $\mathcal{Z}_2=FS_2$, we have
$S_2=(\mbox{det} F) F^{-1}\mbox{sym}(QF^{-1}) F^{-1, T}=
\frac{1}{2}(\det C)^{1/2}C^{-1}D C^{-1},$ hence:
$$ S_2(C,D):D = \frac{1}{2}(\det C)^{1/2} C^{-1}C_t C^{-1}: D
= \frac{1}{2}(\det C)^{1/2}|C^{-1/2}D C^{-1/2}|^2\geq \gamma |D|^2,$$
where $\gamma$ depends on $C$, but is uniform
for $F$ bounded and $\det F$ bounded away from zero.

In \cite{A}, Antman proposes as a sufficient condition for \eqref{strict},
that $S$ be monotone in $D$ in the sense that
$\frac{\partial S}{\partial D}A:A \ge \gamma |A|^2$,
$\gamma>0$,
for $A$ symmetric.
This implies \eqref{strict} under the additional assumption
$S(\cdot,0)\equiv 0$ (i.e. viscous force vanishes at zero velocity), by:
$$ S(C,D):D - S(C,0):D=\int_0^1
\frac{\partial S}{\partial D}(C,\theta D) D:D ~\mbox{d}\theta \ge \gamma
|D|^2.$$
It is easily checked that monotonicity holds for both of the
choices $\mathcal{Z}_1$, $\mathcal{Z}_2$.


\section{Appendix: Phase-transitional elasticity}\label{s:phase}

Another interesting direction for future investigations is the
phase-transitional case, as we now briefly discuss.
%
%
A typical model, as described in \cite{FP}, has the form:
$$ W(F) = | F^T F -C_- |^2\cdot | F^T F -C_+|^2, $$
where:
$$ C_\pm= F_\pm^T F_\pm = \left[\begin{array}{ccc}
1 & 0 & 0\\
0 & 1 & \pm \eps\\
0 & \pm \eps & 1+\eps^2 \end{array}\right],
\qquad
F_\pm = \left[\begin{array}{ccc}
1 & 0 & 0\\
0 & 1 & \pm \eps\\
0 & 0 & 1
\end{array}\right]. $$
Evidently, $W$ is minimized among planar deformation gradients $F$
at the two equilibria $F_\pm$.

A particularly interesting class of solutions to (\ref{visco_eq}) are
stationary phase-transitional shocks connecting the two equilibria
$(b_\pm, a_\pm)$, $a_\pm = (0,\pm\eps,1)$, that is,
zero-speed shocks compatible with the Rankine-Hugoniot condition
for (\ref{hyppar}) with (\ref{dg}): $b_+=b_-$, $DW(a_+) =DW(a_-)$.
Similarly as in section \ref{s:5} and as in 1d case treated by Slemrod
\cite{Sl},   such connections do not exist under
the viscoelastic effects alone, but their existence requires also the inclusion
of third-order surface energy terms as in Section \ref{s:surface}.

%

Solving the system (\ref{5.0}) augmented by the term
$\gamma \mbox{div } \mathcal{E}$ coming from the surface energy
$\mathcal{E}_0$ as in (\ref{exentropy}),
we obtain (for $s=0$) that $b= const$ and
$DW(a)_z= \gamma a_{zzz}$.  Take the capillarity coefficient
$\gamma\neq 0$ and assume that $DW(a_\pm)=0$ (the end-states
in the potential well).
Integrating, we obtain a harmonic oscillator equation:
$DW(a)= \gamma a_{zz}$, with Hamiltonian:
$$ H(a,a_z) = W(a) - \frac{\gamma}{2}|a_z|^2 = const. $$
Hence, a question is whether the connected component of the
level set of $H$, containing $(a_-,0)$, contains also $(a_+,0)$.
Contrary to the 1d case in \cite{Sl}, 
this is only a necessary condition for profile's existence.
Existence or nonexistence of such connections would be an interesting
question for further analytical and numerical investigation.

A second question would be to determine stability of such stationary
transitions, should they exist. We conjecture that, similarly as in
\cite{Z8} for the 1d case, the spectral stability 
follows automatically by energy-considerations.
Stability of nonstationary phase-transitional profiles
has not been treated even in the 1d case,
and would be another interesting problem.




Finally, we point out that the equations with surface energy
do not fit the stability theory of Section \ref{nstab},
since they are third- and not second-order.
However, we expect that the basic methods should still apply, after suitable modifications.
It would be very interesting, for the sake of this and other applications
involving dispersive phenomena (for example, the Hall effect in MHD),
to carry out a complete analysis extending the nonlinear stability
framework to this higher-order case.

\section{Appendix: Nonhyperbolic endstates}\label{s:complex}
Consider the 1D compressible equations (\ref{simplebfirstorder1D})
with the associated profile equation, obtained by setting $a_1=a_2=0$ in (\ref{compode}):
\ba\label{2h}
 2a_3'= (a_3^3-a_3 -\sigma a_3)- (\alpha^3-\alpha -\sigma \alpha),
\qquad
\alpha:= a_{3-} = a_3(-\infty).
\ea
As noticed in section \ref{334},
strict hyperbolicity of (\ref{simplebfirstorder1D}) 
corresponds to $|a_3|>1/\sqrt{3}$.
On the other hand, when $\alpha>a_{3+}>0$,
$\sqrt{(1+\sigma)/3}<\alpha<\sqrt{1+\sigma}$ and $\sigma=a_{3+}^2
+a_{3+}\alpha +\alpha^2>0$, the equation (\ref{2h})
posseses a solution $a_3(z)$, decreasing from its unstable equilibrium
$\alpha=a_3(-\infty)$ to the stable one
$a_{3+}=a_3(+\infty)$. This solution corresponds to a viscous
shock profile of the associated scalar 
equation $a_{3,t} + (a_3^3-a_3)_z=a_{3,zz}$ and the 2d viscous
system:
\begin{equation}\label{v2}
a_{3,t} - b_{3,z} = 0, \qquad b_{3,t}-(a_3^2-a_3)_z = b_{3,zz}.
\end{equation}

Thus, there exist shock profiles of \eqref{simplebfirstorder1D} for which the
left end-state is hyperbolic, with one characteristic greater
than the shock speed $s$ and the other smaller than $s$, but the right end-state is
not hyperbolic as it has
two pure imaginary characteristics with real parts $c_+<s$.
This is a ``complex Lax shock'' of the type considered in \cite{AMPZ,OZ}.

Writing the system (\ref{v2}) in the operator form $(a_3, b_3)^T_t =
Q\cdot (a_3, b_3)^T$ where:
$$Q=\left[\begin{array}{cc}0
    &\partial_z\\(3a_3^2-1)\partial_z&\partial_{zz}
\end{array}\right],$$
we determine the spectrum of the constant solution $(a_+, b_+)$ through
the dispersion relation:
$$ 0=\det\left[\begin{array}{cc} \lambda & -ik\\
-(3a_{3+}^2-1)ik & \lambda+k^2\end{array}\right] =
\lambda^2 + k^2\lambda +(3a_{3+}^2-1)k^2. $$
Consequently:
$$ \lambda(k)\sim  -\frac{1}{2}k^2 \pm k\sqrt{1-3a_{3+}^2} 
\qquad \mbox{ for } ~k\sim 0,$$
yielding a maximal growth rate $e^{rt}$
with $r=Re~ \lambda_{\rm max}(k) \sim (1-3a_{3+}^2)/2$.
As discussed in \cite{AMPZ,OZ}, for speed
$|s|$ sufficiently large compared to the growth rate $r$,
the shock profile can be seen to be
stable, despite instability of its right end-state
as a constant solution.
It also can be shown, by weighted coordinate
techniques as in \cite{Sat,LRTZ}, that
Evans stability together with such a convection vs. growth condition on the
essential spectrum, implies linearized and nonlinear stability
for perturbations that are exponentially localized on the half-line
$z>0$.
Another direction for further investigation might be to
understand whether there are interesting physical
phenomena corresponding to shocks of this type.

In contrast with the situations considered in \cite{AMPZ,OZ},
for which profiles associated with complex Lax shocks are
oscillatory at the complex end,
the profiles here are of ordinary monotone type.
Other complex Lax connections, genuinely two-dimensional,  may
be seen in Fig.  \ref{compressible_phase_portrait_1}(b).
Provided that all characteristics are incoming on the
complex side $z\to +\infty$, nonlinear stability may again be established
by weighted norm methods assuming Evans stability plus an
appropriate convection vs. growth condition as described in \cite{AMPZ,OZ}.

\section{Appendix: Nonexistence of undercompressive profiles} 
\label{nonexist}

We show that undercompressive profiles do not occur
for general shear models as in (\ref{ex23}).
The same argument implies that undercompressive profiles cannot
occur also in the 2D compressible case, for the special class of
potentials $W(a)$ depending only on $|a|^2$.

Note first that every undercompressive profile $(a(z), b(z))$ of
(\ref{shearfirstorder}), with speed $s$, induces an
undercompressive profile of (\ref{shearfirstorderprof}) given by
$z\mapsto a(sz)$ with speed $\sigma=s^2$. 
This statement follows from a more general
observation in \cite{MaZ3, BLZ} relating the type of inviscid shocks
to the connection number for the traveling wave in the reduced ODE (\ref{redprof}).
Alternatively, the same can be checked directly:
if the sum of the number of eigenvalues $\lambda$ for  (\ref{shearfirstorder}) at
$a(-\infty)$ with $\lambda>s$, and the number of eigenvalues $\mu$ at
$a(+\infty)$ with $\mu<s$, is less than $5$ (the dimension of
the system increased by $1$), then the sum of the number of
eigenvalues of (\ref{shearfirstorderprof})
$\lambda^2>\sigma$ and the number of eigenvalues 
$\mu^2<\sigma$ is less than $3$,
when $s>0$.  When  $s<0$ we likewise have that the number of
$\mu^2<\sigma$ plus the number of $\mu^2>\sigma$, is less than $3$.

We shall prove that the system:
\ba\label{shearfirstorderprof1}
a_{t} +( h(|a|^2)a)_z= a_{zz},
\ea
generalizing \eqref{shearfirstorderprof} and (\ref{A3}), where $h=2\nabla\check
\sigma$, admits no undercompressive shocks.

After integrating in $z$, the traveling wave ODE for
(\ref{shearfirstorderprof1}) reads:
\begin{equation}\label{to}
a'=-\sigma a + h(|a|^2)a - \big(-\sigma a_-+ h(|a_-|^2)a_-\big),
\end{equation}
where $\sigma$ is the speed of the shock and $a_-=(a_{1-}, a_{2-})= a(-\infty)$
is its left end-state. Note that all possible right end-states $a_+=(a_{1+},
a_{2+})$ which can be connected to $a_-$, satisfying hence 
the Rankine-Hugoniot condition:
\begin{equation}\label{rhm}
h(|a_-|^2) a_- - \sigma a_- = h(|a_+|^2) a_+ - \sigma a_+,
\end{equation}
must lie on the same line through the origin, due to rotational
invariance of the system (\ref{shearfirstorderprof1}). We may, without
loss of generality, assume it to be the $a_1$ axis, so that:
$a_{2-}=a_{2+}=0$.

The gradient of the right hand side in (\ref{to}):
\begin{equation*}
D\big(h(|a|^2)a-\sigma a\big) =
\big(h(|a|^2)- \sigma\big) \Id + 2\nabla h(|a|^2)a \otimes a
\end{equation*}
has two eigenvalues: $\lambda_1(a) = h(|a|^2) - \sigma + 2\nabla h(|a|^2)|a|^2$,
with the radial direction eigenvector
$a$, and $\lambda_2(a) = h(|a|^2) - \sigma$ with the eigenvector
$a^\perp$ in the transverse (rotational) direction.

Observe further that an undercompressive shock profile must
necessarily be a saddle to saddle connection, and that any saddle point
has one of its invariant manifolds (the one corresponding to
$\lambda_1$) confined to the $a_1$ axis.
We find that the profile must either lie entirely 
on the $a_1$ axis, or else entirely off the axis, leaving
$a_-$ and entering $a_+$ along the transverse direction (orthogonal to
the axis). We now distinguish $3$ cases:

\smallskip

{\bf Case (i) $\mathbf{a_{1+}} \mathbf{a_{1-} >0}$.}
In this situation, in view of (\ref{rhm}), the  transverse eigenvalues $\lambda_2(a_-) =
h(a_{1-}^2) - \sigma$ and $\lambda_2(a_+) = h(a_{1+}^2) - \sigma$ must also have
a common sign. Thus the profiles may only leave or only enter along
the transverse directions, contradicting the assumed behavior.

\smallskip

{\bf Case (ii) $\mathbf{a_{1+}} \mathbf{a_{1-}< 0}$, 
and the profile connection is radial.}
Without loss of generality, assume that $a_{1-}>0$ and $a_{1+}>0$ so
that $a_1'(z)<0$ along the whole profile. In particular, $a_1'(z_0)<0$ at $z_0$
where $a_1(z_0) = 0$.
By (\ref{to}) it follows that:
$$h(a_{1-}^2) a_{1-} - \sigma a_{1-} = -a_1'(z_0)>0,$$
hence $h(a_{1-}^2) - \sigma >0$. This means that the transverse eigenvalue
corresponds to the unstable direction, contradicting the profile being
radial and $a_-$ being a saddle equilibrium point.

\smallskip

{\bf Case (iii) $\mathbf{a_{1+}}\mathbf{a_{1-}<0}$, 
and the profile connection is transverse.}
In this situation $\lambda_2(a_-)>0$ and $\lambda_2(a_+)<0$.
Observe that on the circle $|a|^2=a_{1-}^2$, we have:
$$ a'=\big(h(a_{1-}^2) -\sigma \big) (a-a_-) = \lambda_2(a_-) (a-a_-),$$
and thus the exterior of this circle is a positively invariant set.
Likewise, by (\ref{to}) and (\ref{rhm}), on the circle 
$|a|^2=a_{1+}^2$ there holds:
$$ a'=\big(h(a_{1+}^2) -\sigma \big) (a-a_+) = \lambda_2(a_+) (a-a_+),$$
and hence the exterior of this 
circle is a negatively invariant set.
Consequently, at that of the two saddle points $a_-$ and $a_+$
which has larger norm,  the invariant curve tangent to the
transverse direction must lie, for all times, outside the corresponding
circle.  This is clearly a contradiction and ends the proof, as the
case $a_{1+} = -a_{1-}$ is ruled out by comparing the right hand sides
of the above ODEs.

\end{document}